\documentclass[11pt]{amsart}

\usepackage{amssymb}
\usepackage{amsmath}
\usepackage{amsthm}
\usepackage{enumerate}
\usepackage{algorithmic}
\usepackage[boxed]{algorithm}

\usepackage{color}
\usepackage{float}
\usepackage{tikz-cd}
\usetikzlibrary{arrows}

\usepackage[colorlinks]{hyperref}
\usepackage{mathtools}
\usepackage[english]{babel}

\newtheorem{theorem}{Theorem}[section]
\newtheorem{corollary}[theorem]{Corollary}
\newtheorem{proposition}[theorem]{Proposition}
\newtheorem{lemma}[theorem]{Lemma}

\newtheorem*{algCorrect}{Theorem A.1}

\theoremstyle{definition}
\newtheorem{definition}[theorem]{Definition}

\theoremstyle{remark}
\newtheorem{remark}[theorem]{Remark}
\newtheorem{example}[theorem]{Example}

\newcommand{\Pl}{Pl\"ucker}

\newcommand{\cN}{\mathcal{N}}
\newcommand{\Proj}{\textnormal{Proj}\,}
\newcommand{\Supp}{\textnormal{Supp}}
\newcommand{\Coeff}{\textnormal{Coeff}}

\newcommand{\sat}{{\textnormal{sat}}}
\newcommand{\reg}{\textnormal{reg}}
\newcommand{\Ht}{\textnormal{Ht}}
\newcommand{\Nf}{\mathrm{Nf}}

\newcommand{\Mf}{\mathcal{M}\textnormal{f}}
\newcommand{\hilb}{{\mathcal{H}\textnormal{ilb}}}
\newcommand{\hilbp}{{\mathcal{H}\textnormal{ilb}_{p(t)}^n}}

\newcommand{\PP}{\mathbb{P}}
\newcommand{\Af}{\mathbb{A}}

\newcommand{\SGred}{\xrightarrow{\  sG\ast\  }}
\newcommand{\SGredPar}{\xrightarrow{\ s\mathcal G\ast\ }}

\DeclareMathAlphabet{\mathpzc}{OT1}{pzc}{m}{it}

\begin{document} 

\title[Upgraded methods for marked schemes on a strongly stable ideal]{Upgraded methods for the effective computation of marked schemes on a strongly stable ideal} 

\author[C.~Bertone]{Cristina Bertone}
\address{Dipartimento di Matematica dell'Universit\`{a} di Torino\\ 
         Via Carlo Alberto 10, 
         10123 Torino, Italy}
\email{\href{mailto:cristina.bertone@unito.it}{cristina.bertone@unito.it}}

\author[F.~Cioffi]{Francesca Cioffi}
\address{Dipartimento di Matematica e Applicazioni dell'Universit\`{a} di Napoli Federico II,\\   via Cintia, 80126  Napoli, Italy
         }
\email{\href{mailto:francesca.cioffi@unina.it}{francesca.cioffi@unina.it}}

\author[P.~Lella]{Paolo Lella}
\address{Dipartimento di Matematica dell'Universit\`{a} di Torino\\ 
         Via Carlo Alberto 10, 
         10123 Torino, Italy}
\email{\href{mailto:paolo.lella@unito.it}{paolo.lella@unito.it}}
\urladdr{\url{http://www.personalweb.unito.it/paolo.lella}}

\author[M.~Roggero]{Margherita Roggero}
\address{Dipartimento di Matematica dell'Universit\`{a} di Torino\\ 
         Via Carlo Alberto 10, 
         10123 Torino, Italy}
\email{\href{mailto:margherita.roggero@unito.it}{margherita.roggero@unito.it}}



\begin{abstract}
Let $J\subset S=K[x_0,\ldots,x_n]$ be a monomial strongly stable ideal. The collection $\Mf(J)$ of the homogeneous polynomial ideals $I$, such that the monomials outside $J$ form a $K$-vector basis of $S/I$, is  called a {\em $J$-marked family}. It can be endowed with a structure of affine scheme, called a {\em $J$-marked scheme}. For special ideals $J$, $J$-marked schemes provide an open cover of the Hilbert scheme $\hilbp$, where $p(t)$ is the Hilbert polynomial of $S/J$. Those ideals more suitable to this aim are the \emph{$m$-truncation} ideals $\underline{J}_{\geq m}$ generated by the monomials of degree $\geq m$ in a saturated strongly stable monomial ideal $\underline{J}$. Exploiting a characterization of the ideals in $\Mf(\underline{J}_{\geq m})$ in terms of a Buchberger-like criterion, we compute the equations defining the $\underline{J}_{\geq m}$-marked scheme by a new reduction relation, called {\em superminimal reduction}, and obtain an embedding of  $\Mf(\underline{J}_{\geq m})$ in an affine space of low dimension. In this setting, explicit computations are achievable in many non-trivial cases. Moreover, for every $m$, we give a closed embedding $\phi_m: \Mf(\underline{J}_{\geq m})\hookrightarrow  \Mf(\underline{J}_{\geq m+1})$, characterize those $\phi_m$ that are isomorphisms in terms of the monomial basis of $\underline{J}$, especially we characterize the minimum integer $m_0$ such that $\phi_m$ is an isomorphism for every $m\geq m_0$.
\end{abstract}

\maketitle

\section*{Introduction}

Let $J$ be a monomial ideal of the polynomial ring $S=K[x_0,\ldots,x_n]$ in $n+1$ variables over a field $K$. In this paper,  we refine and develop the study begun in \cite{CR} to characterize the homogeneous polynomial ideals $I\subset S$ such that the monomials outside $J$ form a $K$-vector basis of the $K$-vector space $S/I$. If $J$ is strongly stable, such homogeneous ideals constitute a family $\Mf(J)$, that is called a $J$-marked family and that can be endowed in a very natural way with a structure of affine  scheme, called a $J$-marked scheme, which turns out to be homogeneous with respect to a non-standard grading and flat at $J$ (see \cite{CR}).  Moreover, $J$-marked schemes generalize the notion of Gr\"obner strata \cite{LR} because $\Mf(J)$ contains all the ideals having $J$ as initial ideal with respect to some term order; however in general $\Mf(J)$ contains also ideals which do not belong to a Gr\"obner stratum.

In this paper we focus on a particular class of strongly stable ideals: letting $\underline J$ be a saturated strongly stable ideal, we will consider the truncations ${\underline J}_{\geq m}$, for every positive integer $m$, because in this setting marked schemes give a theoretical and effective alternative to the study of Hilbert schemes as subvarieties of a Grassmannian. Theorem \ref{esempio} and Example \ref{no one-to-one} show the reason for the choice of this special setting.
Let $\hilbp$ be the Hilbert scheme that parameterizes all subschemes of $\mathbb P^n$ with Hilbert polynomial $p(t)$, $r$ be the Gotzmann number of $p(t)$ and $q(t)=\vert S_t\vert -p(t)=\binom{n+t}{n}-p(t)$ be the volume polynomial.
By theoretical results in \cite{BLR,CR,CLMR} we are able to compute first the set $\mathcal B_{p(t)}$ of all saturated strongly stable ideals $\underline J$ in $S$, such that $p(t)$ is the Hilbert polynomial of $S/\underline J$; then, for every ideal $\underline J\in \mathcal B_{p(t)}$, we compute explicit equations of degree $\leq deg(p(t))+2$ defining $\Mf({\underline J}_{\geq r})$ as an affine subscheme of $\Af^{p(r)q(r)}$. In particular, every $\Mf({\underline J}_{\geq r})$ can be embedded in $\hilbp$ as an open subscheme and moreover, as $\underline J$ varies in $ \mathcal B_{p(t)}$, the ${\underline J}_{\geq r}$-marked schemes $\Mf({\underline J}_{\geq r})$ form an open cover of $\hilbp$, up to  changes of coordinates in $\PP^n$. Observe that this is not true for $\Mf(\underline J)$, because in general $\Mf(\underline J)$ is not isomorphic to an open subset of $\hilbp$ (see Example \ref{espunti} and \cite[Section 5]{RT}).
Such computational method is effective because the dimension  ${p(r)q(r)}$  of the affine space in which the ${\underline J}_{\geq r}$-marked schemes $\Mf({\underline J}_{\geq r})$ are embedded is significantly lower than the number $\binom{\vert S_r\vert }{q(r)}$ of Pl\"ucker coordinates. However there is room for further significant improvements.

The present paper is inspired by two questions raised, on the one hand, by similarities of marked schemes with Gr\"obner strata and, on the other hand, by experimental observations on examples.

First, {we observed that   we could eliminate  a significant number of variables from the equations defining $\Mf({\underline J}_{\geq m})$ as an affine subscheme of $\Af^{p(m)q(m)}$, computed using the method developed in \cite{CR}; in this way we obtain equations of higher degree than the starting ones,}
but often more convenient to  use (for example, see \cite[Appendix]{CR}). This feature has already been observed and studied for Gr\"obner strata in \cite{LR}. The bottleneck is that elimination of variables is too time-consuming. From this we wondered how to obtain this new set of equations using in the computations only necessary variables, avoiding the elimination process.

Our second observation is that, for a fixed $\underline J\in \mathcal B_{p(t)}$, as the integer $m$ grows, the families parameterized by marked schemes $\Mf({\underline J}_{\geq m})$ become larger, up to a certain value of $m$ bounded by $r$. The study of relations among marked schemes $\Mf({\underline J}_{\geq m})$ as $m$ varies can improve the efficiency of the computational methods in \cite{CR}: indeed, if $\Mf({\underline J}_{\geq m})$ and $\Mf({\underline J}_{\geq m'})$ are isomorphic for some integers $m'<m$, then we can choose to compute defining equations that involve a lower number of variables, that is equations for  $\Mf({\underline J}_{\geq m'})\subseteq \Af^{p(m')q(m')}$. In particular, for applications to the study of Hilbert schemes, we would like to  determine a priori the minimum integer $m_0$ for which $\Mf({\underline J}_{\geq m_0})$ is isomorphic to an open subset of $\hilbp$, that is $\Mf({\underline J}_{\geq m_0})\simeq \Mf({\underline J}_{\geq r})$.

In this paper, considering truncated ideals ${\underline J}_{\geq m}$, we answer to both questions by a new reduction algorithm,  called \textit{superminimal reduction}, that uses, for every $I\in \Mf({\underline J}_{\geq m})$, its \textit{${\underline J}_{\geq m}$-superminimal basis} (see Definition \ref{superminimalbasis}), a special subset of the ${\underline J}_{\geq m}$-marked basis of $I$.

For every  strongly stable monomial ideal $J$, the  notion of $J$-marked basis (Definition \ref{def:defJbase}) is the main tool for the study of marked schemes in \cite{CR} and also the starting point of the  present paper. Indeed  a homogeneous ideal $I$ belongs to $\Mf(J)$ if and only if $I$ is generated by a $J$-marked basis $G$ (Proposition \ref{cor1}). This basis resembles a reduced Gr\"obner basis for $I$, where $J$ plays the role of the initial ideal and the strongly stable property plays the role of the term order.

Indeed, similarly to a reduced Gr\"obner basis, $G$ is a system of generators of $I$ that contains a polynomial $f_\alpha$ for every term $x^\alpha$ in the monomial basis of $J$: $f_\alpha=x^\alpha-T(f_\alpha)$ where no monomial appearing in $T(f_\alpha)$ belongs to $J$. Moreover, $G$ is characterized by a Buchberger-like criterion (Theorem \ref{BuchCrit1}) and allows to compute the $J$-reduced form modulo $I$ of every polynomial in $S$, by a Noetherian reduction process (Proposition \ref{construction of $J$-normal form}).

The $J$-superminimal basis of $I$ introduced in the present paper is a special subset $sG$ of $G$ containing a polynomial for every term in the monomial basis of the saturated ideal $\underline J$ (for the details, see Definitions \ref{defsupermin} and \ref{superminimalbasis}): the two sets $G$ and $sG$ are equal if and only if $\underline J=J$. Using only polynomials in $sG$ and the strongly stable property of $J$, we define a special process of reduction $\SGred $ called superminimal reduction.

\medskip

In the special case when $J$ is a truncation ${\underline J}_{\geq m}$ of a saturated strongly stable ideal $\underline J$, the ${\underline J}_{\geq m}$-superminimal basis $sG$ has very interesting properties. First of all in this case (but not in general) the superminimal reduction $\SGred$ turns out to be Noetherian (see Theorem \ref{ridsm}, (\ref{ridsm-i}) and Example \ref{no ssr}).
Moreover, although in general $sG$ is not a system of generators of $I$, it completely determines the ideal $I$  because we can solve the ideal-membership problem by the superminimal reduction process $\SGred$ (Theorem \ref{ridsm}, (\ref{ridsm-iv})). This allows to compute equations for $\Mf({\underline J}_{\geq m})$ as a subscheme of an affine space of dimension far lower than $p(m)q(m)$, without any variable elimination process (Theorem \ref{smallaffine}), answering the first question above.

In this new setting, in Theorem \ref{schisom} we compare the ${\underline J}_{\geq m}$-marked schemes $\Mf({\underline J}_{\geq m})$ for a fixed saturated $\underline J$ as $m$ varies, using superminimal bases. We  prove that for every $m$ there is a closed scheme-theoretical embedding $\phi_m: \Mf({\underline J}_{\geq m-1})\hookrightarrow  \Mf({\underline J}_{\geq m})$. Moreover, we provide an easy criterion on the monomial basis of $\underline J$ to characterize the integers $m$ for which $\phi_m$ is an isomorphism. Especially, this criterion allows to determine the minimum integer $m_0$ such that $\phi_m$ is an isomorphism for every $m\geq m_0$, and in particular $\Mf({\underline J}_{\geq m_0})$ is isomorphic to an open subset of $\hilbp$ (see \cite{BLR}).

\medskip

Our investigation on marked schemes lies in the framework  of the methods and results obtained in the last years by several authors \cite{CF,FR,LR,NS,Ro,RT} about families of ideals with a fixed monomial basis for the quotient. Another close framework is the one in \cite{BM,MoraR}, where the authors study the collection of all monomial ideals $J$ that are initial ideals of a fixed homogeneous ideal $I$ w.r.t. some term order.

In \cite{BLR}, the results of \cite{CR} and the ones of the present paper are applied to study relations among marked schemes and Hilbert schemes; in particular,  in \cite{BLR} the authors study how marked schemes can be used to obtain a computable open cover of $\hilbp$ that has also interesting theoretical features.
We are confident that these results, both theoretical and computational ones, may be helpful in the solution of some open problems about Hilbert schemes; {indeed, they have been already applied   in order to investigate the locus of points of the Hilbert scheme with bounded regularity (see \cite{BBR}); the ideas and strategies used by \cite{GFR}  to study deformations of ACM curves are inspired by the ones in the present paper;  in \cite{LS} the authors apply the computational strategy  to the Hilbert scheme of locally CM curves. Other investigations led by these tools are in progress.
 In the future, we are interested in deeply comparing marked bases with other kinds of Gr\"obner-like bases, referring to \cite{Mo2}.}

\medskip
In Section \ref{sec:notation}, we introduce notations and basic results and, in Section \ref{buch}, we recall the Buchber-ger-like criterion described in \cite{CR}, with some development that involves the Eliahou and Kervaire syzygies of a strongly stable ideal (Theorem \ref{BuchCrit1}, (\ref{buchcrit1-iii}) and Corollary \ref{sizEK}). Moreover, we compute sets of generators of the ideal $\mathfrak A_J$ that defines the structure of affine scheme of $\Mf(J)$ (see Corollary \ref{last} and Remark \ref{rem5}).

In Section \ref{secsupermin}, we define the superminimal reduction (Definition \ref{sminred}) and investigate its properties.  In Section \ref{buchsec} we describe a new Buchberger-like criterion for ${\underline J}_{\geq m}$-marked bases (Theorem \ref{nostrobuch}) and some variants of it (Corollary \ref{nostrobuchEK} and Theorem \ref{paolo}). In particular, the second variant leads to a remarkable improvement of the efficiency of explicit computational procedures.

In Section \ref{ctildate}, we focus on the ideal 
 that defines the structure of affine scheme of $\Mf({\underline J}_{\geq m})$ and we characterize the integers $m,m'$, $m>m'$, such that the schemes $\Mf({\underline J}_{\geq m})$ and $\Mf({\underline J}_{\geq m'})$ are isomorphic (Theorem \ref{schisom}).

Finally, in Section \ref{secexamples} we provide examples in which we apply the proved results and we  compute the equations defining the affine structure of a ${\underline J}_{\geq m}$-marked scheme in a \lq\lq small\rq\rq \ affine space, using the Algorithm that we describe  in the Appendix.

\section{Notations and generalities} \label{sec:notation}

Let $K$ be an algebraically closed field and $S:=K[x_0,\dots,x_n]$ ($K[x]$ for short) the polynomial ring in $n+1$ variables with $x_0<\dots<x_n$. We will denote by $x^\alpha=x_0^{\alpha_0} \cdots x_n^{\alpha_n}$ every monomial in $S$, where $\alpha=(\alpha_0,\ldots,\alpha_n)$ is its multi-index and $\vert\alpha\vert$ is its degree.

We say that a monomial $x^\gamma$ is divisible by $x^\alpha$ ($x^\alpha \mid x^\gamma$ for short) if there exists a monomial $x^\beta$ such that $x^\alpha \cdot x^\beta = x^\gamma$. If such monomial does not exist, we will write $x^\alpha \nmid x^\gamma$.  For every monomial $x^\alpha\not= 1$, we set $\min(x^\alpha):= \min\{x_i : x_i \mid x^\alpha\}$ and $\max(x^\alpha):= \max\{x_i : x_i \mid x^\alpha\}$.

We will denote by $>_{\tt{Lex}}$ the usual lexicographic order on the monomials of $S$: in our setting $x^\alpha>_{\tt{Lex}}x^\beta$ if the last non-null element of $\alpha-\beta$ is positive.

We consider the standard grading on $S = \bigoplus_{m\in\mathbb{Z}} S_m$, where $S_m$ is the additive group of homogeneous polynomials of degree $m$; we let $S_{\geqslant m} =\bigoplus_{m'\geq m} S_{m'}$ and in the same way, for every subset $A\subseteq S$, we let $A_m=A\cap S_m$ and $A_{\geqslant m}=A\cap S_{\geqslant m}$. Elements and ideals in $S$ are always supposed to be homogeneous.
\smallskip

We will say that a monomial $x^\beta$ can be obtained by a monomial $x^\alpha$ through an \emph{elementary  move} if $x^\alpha x_j = x^\beta x_i$ for some variables $x_i\neq x_j$. In particular, if $i < j$, we say that $x^\beta$ can be obtained by $x^\alpha$ through an \emph{increasing} elementary move and we write $x^\beta = \mathrm{e}_{i,j}^{+}(x^\alpha)$, whereas if $i > j$ the move is said to be \emph{decreasing} and we write $x^\beta = \mathrm{e}_{i,j}^{-} (x^\alpha)$. The transitive closure of the relation $x^\beta > x^\alpha$ if $x^\beta = \mathrm{e}_{i,j}^{+}(x^\alpha)$  gives a partial order on the set of monomials of a fixed degree, that we will denote by $>_{B}$ and that is often called \emph{Borel partial order}:
\[
x^\beta  >_{B} x^\alpha\ \Longleftrightarrow\ \exists\  x^{\gamma_1}, \dots, x^{\gamma_t} \text{ such that } x^{\gamma_1} = \mathrm{e}^+_{i_0,j_0}(x^\alpha),\    \dots\ ,x^\beta = \mathrm{e}^{+}_{i_t,j_t}(x^{\gamma_t})
\]
for suitable indexes $i_k,j_k$.
In analogous way, we can define the same relation using decreasing moves:
\[
x^\beta  >_{B} x^\alpha\ \Longleftrightarrow\ \exists\  x^{\delta_1}, \dots, x^{\delta_s}  \text{ such that } x^{\delta_1} = \mathrm{e}^{-}_{h_0,l_0}(x^\beta),\ \dots\ ,x^\alpha = \mathrm{e}^{-}_{h_s,l_s}(x^{\delta_s})
\]
for suitable indexes $i_k,j_k$.
Note that every term order $\succ$ is a refinement of the Borel partial order $>_B$, that is $x^\beta >_B x^\alpha$ implies that $x^\beta \succ x^\alpha$.

\begin{definition}\label{definBorel}
An ideal $J \subset K[x]$ is said to be \emph{strongly stable} if every monomial $x^\beta$ such that $x^\beta >_{B} x^\alpha$, with $x^\alpha \in J$, belongs to $J$.
\end{definition}

A strongly stable ideal is always \emph{Borel fixed}, that is fixed by the action of the Borel subgroup of upper triangular matrices of $GL(n+1)$. If $ch(K)=0$, also the vice versa holds (e.g. \cite{D}) and \cite{Gall} guarantees that in generic coordinates the initial ideal of an ideal $I$, w.r.t. a fixed term order, is a constant Borel fixed monomial ideal called the {\em generic initial ideal} of $I$.

If $J$ is a monomial ideal in $S$, $B_J$ will denote its monomial basis and $\cN(J)$ its \emph{sous-escalier}, that is the set of monomials not belonging to $J$.

An homogeneous ideal $I$ is $m$-{\it regular} if the $i$-th syzygy module of $I$ is generated in degree $\leq m+i$, for all $i\geq 0$. The {\it regularity}  of $I$
is the smallest integer $m$ for which $I$ is $m$-regular;  {we denote it by $\reg(I)$}. The {\it saturation} of a homogeneous ideal $I$ is $I^{\sat}=\{f\in S\ \vert \ \ \forall \ j=0,\ldots,n,\exists \ r \in {\mathbb N} : x_j^r f \in I\}$. The ideal $I$ is {\it saturated} if $I^{\sat}=I$ and is {\it $m$-saturated} if $(I^{\sat})_t = I_t$ for each $t\geq m$. The {\em satiety}  of $I$ is the smallest integer $m$ for which $I$ is $m$-saturated; {we denote it by $\sat(I)$}.

We recall that if $J$ is strongly stable then $\reg(J) = \max \{\deg x^\alpha \ : \ x^\alpha \in B_J\}$ \cite[Proposition 2.9]{BS} and $\sat(J)=\max \{\deg x^\alpha \ : \ x^\alpha \in B_J \text{ and } x_0 \mid x^\alpha\}$ (for example, see \cite[Corollary 2.10]{Gr}).

\begin{lemma} Let $J$ be a strongly stable ideal in $K[x_0, \dots, x_n]$. Then:
\begin{enumerate}[(i)]
\item $x^\alpha \in J\setminus B_J\ \Rightarrow\ \dfrac{x^\alpha}{\min(x^\alpha)} \in J$;
\item $x^\beta \in \cN(J)$ and  $x_ix^\beta \in J \ \Rightarrow$ either $x_ix^\beta \in B_J$ or $x_i > \min(x^\beta)$.
\end{enumerate}
\end{lemma}

\begin{proof}
Both properties follow from Definition \ref{definBorel}.
\end{proof}

\begin{definition}\label{satmon}
For every monomial $x^\alpha$ in $S$ we denote by $x^{\underline{\alpha}}$ the monomial obtained putting $x_0=1$. Analogously, if $J$ is a monomial ideal in $K[x]$, we denote by $\underline J$ the ideal in $K[x]$ generated by $\{x^{\underline\alpha}\ : \ x^\alpha \in B_J\}$.
\end{definition}

If $J$ is strongly stable, then ${J}^{\sat}=\underline J$ (this follows straightforwardly from \cite[Corollary 2.10]{Gr}); in particular, the set  $\{x^{\underline\alpha}\ : \ x^\alpha \in B_J\}$ of the monomials $x^{\underline \alpha}$, such that $x^\alpha=x^{\underline \alpha}\cdot x_0^t$ belongs to $B_J$ for a suitable $t\geq 0$, contains the monomial basis $B_{\underline J}$.

Many tools we are going to use were introduced in \cite{RStu} and developed in \cite{CR}. For this reason,  we now resume some notations and definitions given in those papers.

\begin{definition}
For any non-zero homogeneous polynomial $f \in S$, the \textit{support} of $f$ is the set $\Supp(f)$ of monomials that appear in $f$ with a non-zero coefficient.
\end{definition}

\begin{definition}[\cite{RStu}]
A \emph{marked polynomial} is a polynomial $f\in S$ together with a specified monomial of $\Supp(f)$ that will be called \emph{head term} of $f$ and denoted by $\Ht(f)$.
\end{definition}

\begin{remark}
Although in this paper we use the word ``monomial'', we say ``head term'' for coherency with the notation introduced by \cite{RStu}. Anyway, in this paper there will be no possible ambiguity on the meaning of \ ``head term of $f$'', because we will always consider marked polynomials $f$ such that the coefficient of $\Ht(f)$ in $f$ is 1.
\end{remark}

\begin{definition}[\cite{CR}]
Given a monomial ideal $J$ and an ideal $I$, a polynomial is {\em $J$-reduced} if its support is contained in $\cN(J)$ and a {\em $J$-reduced form modulo $I$} of a polynomial $h$ is a polynomial $h_0$ such that $h-h_0\in I$ and $\Supp(h_0) \subseteq \cN(J)$. If {there is a unique} $J$-reduced form modulo $I$ of $h$, we call it \emph{$J$-normal form modulo $I$} and denote it by $\Nf(h)$.
\end{definition}

Note that every polynomial $h$ has a unique $J$-reduced form modulo an ideal $I$ if and only if $\cN(J)$ is a $K$-basis for the quotient $S/I$ or, equivalently, $S=I\oplus \langle \cN (J) \rangle $ as a $K$-vector space. If moreover $I$ is homogeneous, the $J$-reduced form modulo $I$ of a homogeneous polynomial is supposed to be homogeneous too. These facts motivate the following definitions.

\begin{definition}\label{def:defJbase}
A finite set $G$ of homogeneous marked polynomials $f_\alpha=x^\alpha-\sum c_{\alpha\gamma} x^\gamma$, with  $\Ht(f_\alpha)=x^\alpha$, is called a $J$-\emph{marked set} if the head terms $\Ht(f_\alpha)$ form the monomial basis $B_J$ of a monomial ideal $J$, are pairwise different and every $x^\gamma$ belongs to $\cN(J)$, i.e.  $\vert\Supp(f_\alpha)\cap J \vert =1$. We call \emph{tail} of $f_\alpha$ the polynomial $T(f_\alpha):=\Ht(f_\alpha)-f_\alpha$, so that $\Supp(T(f_\alpha))\subseteq \mathcal N(J)$.  A $J$-marked set $G$ is a $J$-\emph{marked basis} if $\cN(J)$ is a basis of $S/(G)$ as a $K$-vector space.
\end{definition}

\begin{definition}
The   collection of all the homogeneous ideals $I$ such that $\cN(J)$ is a basis of the quotient $S/I$ as a $K$-vector space will be denoted by $\Mf(J)$ and called a $J$-\emph{marked family}. If $J$ is a strongly stable ideal, then $\Mf(J)$ can be endowed with a natural structure of scheme (see \cite[Section 4]{CR}) that we call $J$-\emph{marked scheme}.
\end{definition}

\begin{remark}\label{Jmarkedbasis}\
\begin{enumerate}[(i)]
\item \label{rk:Jmarkedbasis_i} The ideal $(G)$ generated by a $J$-marked basis $G$ has the same Hilbert function of $J$, hence $\dim_K J_m = \dim_K (G)_m$, by the definition of $J$-marked basis itself. Moreover, note that a $J$-marked basis is unique for the ideal that it generates, by the uniqueness of the $J$-normal forms modulo $I$ of the monomials in $B_J$.
\item \label{rk:Jmarkedbasis_ii} $\Mf(J)$ contains every homogeneous ideal having $J$ as initial ideal w.r.t. some term order, but it might also contain other ideals: see \cite[Example 3.18]{CR}.
\item \label{rk:Jmarkedbasis_iii} When $J$ is a strongly stable ideal, all homogeneous polynomials have $J$-reduced forms modulo every ideal generated by a $J$-marked set $G$ (see \cite[Theorem 2.2]{CR}).
\end{enumerate}
\end{remark}

\begin{proposition}\label{cor1} Let $J$ be a strongly stable ideal, $I$ be a homogeneous ideal generated by a $J$-marked set $G$. The following facts are equivalent:
\begin{enumerate}[(i)]
\item\label{it:cor1_i}  $I\in \Mf(J)$
\item\label{it:cor1_ii} $G$ is a $J$-marked basis;
\item\label{it:cor1_iii} $\dim_kI_{t}=\dim_K J_{t}$, for every integer $t$;
\item \label{it:cor1_iv} if $h \in I$ and $h$ is $J$-reduced, then $h=0$.
\end{enumerate}
\end{proposition}

\begin{proof} For the equivalence among the first three statements, see \cite[Corollaries 2.3, 2.4, 2.5]{CR}. For the equivalence among (\ref{it:cor1_i}) and (\ref{it:cor1_iv}), observe that if $I\in \Mf(J)$, then every polynomial has a unique $J$-reduced form modulo $I$; so, the $J$-reduced form modulo $I$ of a polynomial of $I$ must be null. Vice versa, it is enough to show that every polynomial $f$ has a unique $J$-reduced form modulo $I$. Let $\bar f$ and $\bar{\bar f}$ be two $J$-reduced forms modulo $I$ of $f$. Then, $\bar f -\bar{\bar f}$ is a $J$-reduced polynomial of $I$ because $f-\bar f$ and $f-\bar{\bar f}$ belong to $I$ by definition. We are done, because $\bar f -\bar{\bar f}$ is null by the hypothesis.
\end{proof}

\section{\texorpdfstring{Background on Buchberger-like criterion for $J$-marked bases and some developments}{Background on Buchberger-like criterion for J-marked bases and some developments}}\label{buch}

In this section we recall and develop some results of \cite{CR}. Throughout this section, {\em $J$ is a strongly stable ideal and $G$ is a $J$-marked set}.

\begin{definition}\label{def:$V_m$}
Let $m_J:=\min\{t : J_t\neq (0)\}$ be the initial degree of $J$. For every $\ell\geq m_J$ we define the set
 \[
 W_\ell:=\{x^\delta f_\alpha \ \vert \  f_\alpha \in G \text{ and } \ \vert \ \delta+\alpha\vert=\ell\}
 \]
that becomes a set of marked polynomials by letting $\Ht(x^\delta f_\alpha)=x^{\delta+\alpha}$. We set $W=\cup_\ell W_\ell$.
For every $\ell\geq m_J$ we also define a special subset of $W_\ell$:
\[
V_\ell:=\{x^\delta f_\alpha \in W_\ell \ \vert \ x^\delta=1 \text{ or }\max(x^\delta)\leq \min(x^\alpha)\}.
\]
We let $V=\cup_\ell V_\ell$. Moreover, $\langle V\rangle$ denotes the vector space generated by the polynomials in $V$ and $\stackrel{V_\ell}\longrightarrow$ is the reduction relation on homogeneous polynomials of degree $\ell$ defined in the usual sense of Gr\"obner basis theory (see also \cite[Proposition 3.6]{CR}).
\end{definition}

The above Definition is equivalent to the Definition 3.2 in \cite{CR} due to Remark 3.3 of the same paper.

Note that $(G)_\ell$ is generated by $W_\ell$ as a $K$-vector space.

\begin{lemma}\label{forseteo}
Let $J$ be a strongly stable ideal. An ideal $I$ generated by a $J$-marked set $G$ belongs to $\Mf(J)$ if and only if $\langle W\rangle=\langle V\rangle $ as $K$-vector spaces.
\end{lemma}
\begin{proof}
It is sufficient to observe that for every $\ell\geq m_J$, the number of elements in $V_{\ell}$ is equal to the number of monomials in $J_{\ell}$, so $\dim \langle V_\ell\rangle \leq \dim J_\ell$. On the other hand, $\dim \langle W_\ell \rangle=\dim I_\ell\geq \dim J_\ell$ by \cite[Corollary 2.3]{CR}. By Proposition \ref{cor1} we get the equivalence of the statements.
\end{proof}

We have already recalled that, when $J$ is a strongly stable ideal, every homogeneous polynomial has a $J$-reduced form modulo an ideal generated by a $J$-marked set $G$ (Remark \ref{Jmarkedbasis} (\ref{rk:Jmarkedbasis_iii})). Further, a $J$-reduced form of a homogeneous polynomial can be constructed by the reduction relation $\stackrel{V_\ell}\longrightarrow$, as it is recalled by next Proposition.

\begin{proposition} \label{construction of $J$-normal form} {\rm \cite[Proposition 3.6]{CR}}
With the above notation, every monomial $x^\beta\in J_\ell$ can be reduced to a $J$-reduced form modulo $(G)$ in a finite number of reduction steps, using only polynomials of $V_\ell$. Hence, the reduction relation $\stackrel{V_\ell}\longrightarrow$ is Noetherian.
\end{proposition}

The Noetherianity of the reduction relation $\stackrel{V_\ell}\longrightarrow$ provides an algorithm that reduces every homogeneous polynomial of degree $\ell$ to a $J$-reduced form modulo $(G)$ in a finite number of steps. We note that on the one hand it is convenient to substitute the polynomials in $V_\ell$ by their $J$-reduced normal forms for an efficient implementation of a reduction algorithm, but, on the other hand, in the proofs it is convenient to use the polynomials of $V_\ell$ as constructed in Definition \ref{def:$V_m$}.

\subsection{\texorpdfstring{Order on $W_\ell$}{Order of W\_l}}

Using the Noetherianity of the reduction relation $\stackrel{V_\ell}\longrightarrow$, we can recognize when a $J$-marked set is a $J$-marked basis by a Buchberger-like criterion (see \cite[Theorem 3.12]{CR}). To this aim we need to set an order on the set $W_\ell$.

The order that we are going to define on $W_\ell$ in Definition \ref{order2} is based on the following Definition and Lemma that are inspired by \cite{EK} and \cite[Lemma 2.11]{MS}.

\begin{definition}\label{stellina}
Given a strongly stable monomial ideal $J$ in $S$, with monomial basis $B_J$, and a monomial $x^\gamma\in J$, we define
\[
x^\gamma=x^\alpha \ast_J x^\eta, \quad \text{ with } \gamma=\alpha+\eta, \ x^\alpha \in B_{J}\text{ and }\min(x^\alpha)\geq \max (x^\eta).
\]
This decomposition exists and is unique (see \cite[Lemma 1.1]{EK}).
\end{definition}

\begin{lemma}\label{descLex}
Let $J$ be a strongly stable ideal. If $x^\epsilon$ belongs to $\cN(J)$ and $x^{\epsilon}\cdot x^\delta=x^{\epsilon+\delta}$ belongs to $J$ for some $x^\delta$, then $x^{\epsilon+\delta}=x^{\alpha}\ast_J x^{\eta}$ with $x^{\eta}<_{\mathtt{Lex}}x^\delta$.  Furthermore:
\begin{enumerate}[(i)]
\item if $\vert \delta\vert=\vert\eta\vert$, then $x^{\eta}<_{B}x^\delta$; and
\item $x^{\underline{\eta}}<_{\mathtt{Lex}}x^{\underline{\delta}}$.
\end{enumerate}
\end{lemma}

\begin{proof}
We can assume that $x^{\delta}$ and $x^{\eta}$ are coprime; indeed, if this is not the case, we can divide the involved equalities of monomials by  $\gcd(x^{\delta},x^{\eta})$. If $x^\eta=1$, all the statements are obvious. If $x^\eta\neq 1$, then $\min(x^\delta) \vert x^\alpha$ because $x^{\delta}$ and $ x^{\eta}$ are coprime, hence $ \min(x^\delta) \geq \min(x^\alpha)\geq \max(x^\eta)$ and so $\min(x^\delta)> \max(x^\eta)$ because they cannot coincide. This inequality implies both $x^{\eta}<_{\mathtt{Lex}}x^\delta$ and $x^{\underline{\eta}}<_{\mathtt{Lex}}x^{\underline{\delta}}$. Moreover, if $\vert \delta\vert=\vert\eta\vert$, this is also sufficient to conclude that $x^{\eta}<_{B}x^\delta$.
\end{proof}

\begin{remark} \label{rm:V}
Observe that if $g_\beta=x^\delta f_\alpha$ belongs to $V_\ell$, then $x^\beta=x^\alpha\ast_J x^\delta$.
\end{remark}

\begin{definition}\label{order2}
Let $\geq$ be any order on $G$  and $x^\delta f_\alpha$, $x^{\delta'}f_{\alpha'}$ be two elements of $W_\ell$. We set
\[
x^\delta f_\alpha \succeq_\ell x^{\delta'}f_{\alpha'} \Leftrightarrow
x^\delta >_{\tt{Lex}} x^{\delta'} \text{ or } x^\delta = x^{\delta'} \text{ and } f_\alpha\geq f_{\alpha'}.
\]
\end{definition}

\begin{lemma}\label{reduction2} \
\begin{enumerate}[(i)]
\item \label{reduction2-i}For every two elements $x^\delta f_\alpha$, $x^{\delta'}f_{\alpha'}$ of $W_\ell$ we get 
\[
x^\delta f_\alpha \succeq_\ell x^{\delta'}f_{\alpha'} \Rightarrow \forall x^\eta :\ \  x^{\delta+\eta} {f_\alpha} \succeq_{\ell'} x^{\delta'+\eta}f_{\alpha'},
\]
 where $\ell'=\vert \delta +\eta + \alpha\vert$.
\item \label{reduction2-ii}Every polynomial $g_\beta\in V_\ell$ is the minimum w.r.t. $\preceq_\ell$ of the subset $W_\beta$ of $W_\ell$ containing all polynomials of $W_\ell$ with $x^\beta$ as head term.
\item \label{reduction2-iii}If $x^\delta f_\alpha$ belongs to $W_\ell\setminus G_\ell \text{ and } x^\beta$ belongs to $\Supp(x^\delta T(f_\alpha)) \text{ with } \ g_\beta\in V_\ell$, then $x^\delta f_{\alpha} \succ_\ell g_\beta$.
\end{enumerate}
\end{lemma}

\begin{proof} \
\begin{enumerate}[(i)]
\item  This follows by the analogous property of the term order $>_{\tt{Lex}}$.
\item Let $g_\beta=x^{\delta'}f_{\alpha'}$ be the polynomial of $V$ such that $x^\beta=x^{\alpha'}\ast_J x^{\delta'}$ and $x^\delta f_\alpha$ be another polynomial of $W_{\beta}$. We can assume that $x^{\delta}$ and $x^{\delta'}$ are coprime; otherwise, we can divide the involved inequalities of monomials by  $\gcd(x^{\delta},x^{\delta'})$. By Remark \ref{rm:V} and Definition \ref{stellina}, we have that $\max(x^{\delta'})\leq \min(x^{\alpha'})$ and $\max(x^{\delta})> \min(x^{\alpha})$. Then, we get $\max(x^\delta) > \max(x^{\delta'})$ because $x^{\alpha'}\nmid x^\alpha$ and $x^{\alpha}\nmid x^{\alpha'}$. Thus, $x^\delta >_{\mathtt{Lex}} x^{\delta'}$.
\item If $x^\beta$ belongs to $B_J$ we are done. Otherwise, let $x^\beta=x^{\alpha'}\ast_J x^{\delta'}$ and note that every monomial of $\Supp(x^\delta f_\alpha)$ is a multiple of $x^\delta$, in particular $x^{\beta}=x^{\delta+\gamma}$ for some $x^\gamma \in \cN(J)$. By Lemma \ref{descLex}, we get $x^{\delta'}<_{\tt{Lex}}x^\delta$. \qedhere
\end{enumerate}
\end{proof}

\begin{remark}
We point out that the order defined in \cite[Definition 3.9]{CR} does not satisfy the conditions listed in Lemma \ref{reduction2} and in \cite[Lemma 3.10]{CR}. These conditions have a crucial role  in the proof of \cite[Theorem 3.12]{CR} and for this reason it has been a mistake to use the order of \cite[Definition 3.9]{CR} in that Theorem. So, here we replace such order by that defined in new Definition \ref{order2}. Aside the order, the original reduction and Buchberger criterion are the same, as we will state in Theorem \ref{BuchCrit1}, (\ref{buchcrit1-i}) and (\ref{buchcrit1-ii}). Also, we give an improvement by Theorem \ref{BuchCrit1}, (\ref{buchcrit1-iii}) and by Corollary \ref{sizEK}. Moreover, we observe that for the same reason the results about syzygies of the ideal $I$ generated by a $J$-marked basis proposed in \cite[Section 3]{CR} hold by using the order on $W_\ell$ of Definition \ref{order2} and do not hold by using the order of \cite[Definition 3.9]{CR}.
\end{remark}

\subsection{\texorpdfstring{Improved Buchberger-like criterion for $J$-marked bases}{Improved Buchberger-like criterion for J-marked bases}}
\begin{definition}
The {\em S-polynomial} of two elements $f_\alpha$, $f_{\alpha'}$ of a $J$-marked set $G$ is the polynomial $S(f_\alpha, f_{\alpha'}):=x^\gamma f_\alpha-x^{\gamma'}f_{\alpha'}$, where $x^{\gamma+\alpha}=x^{\gamma'+\alpha'}=lcm(x^\alpha,x^{\alpha'})$.
\end{definition}

\begin{theorem}{\rm (Buchberger-like criterion)}\label{BuchCrit1}
Let $J$ be a strongly stable ideal and $I$ the homogeneous ideal generated by a $J$-marked set $G$. With the above notation, TFAE:
\begin{enumerate}[(i)]
\item \label{buchcrit1-i}$I\in \Mf(J)$;
\item \label{buchcrit1-ii}$ \forall f_\alpha, f_{\alpha'} \in G,\  S(f_\alpha, f_{\alpha'}) \stackrel{V_\ell}\longrightarrow 0$;
\item \label{buchcrit1-iii}$ \forall f_\alpha, f_{\alpha'} \in G,\  S(f_\alpha, f_{\alpha'})=x^\gamma f_\alpha-x^{\gamma'}f_{\alpha'} =\sum a_j x^{\eta_j}f_{\alpha_j}$, with $x^{\eta_j}<_{\tt{Lex}}\max_{\tt{Lex}}\{x^\gamma,x^{\gamma'}\}$ and $x^{\eta_j}f_{\alpha_j}\in V_\ell$.
\end{enumerate}
\end{theorem}

\begin{proof}
For the equivalence between (\ref{buchcrit1-i}) and (\ref{buchcrit1-ii}), we refer to the proof of \cite[Theorem 3.12]{CR} by using Definition \ref{order2} instead of \cite[Definition 3.9]{CR}.

Statement (\ref{buchcrit1-ii}) implies (\ref{buchcrit1-iii}) by the definition of the reduction relation $\stackrel{V_\ell}\longrightarrow$ and by Lemma \ref{reduction2} (\ref{reduction2-iii}).
It remains to prove that statement (\ref{buchcrit1-iii}) implies (\ref{buchcrit1-i}).

We want to prove that $I=\langle V \rangle$ or, equivalently by Lemma \ref{forseteo}, that $\langle V\rangle=\langle W\rangle$.  It is sufficient to prove that $x^\eta\cdot V\subseteq \langle V\rangle$, for every monomial  $x^\eta$. We proceed by induction on the monomials $x^\eta$, ordered according to $\tt{Lex}$. The thesis is obviously true for $x^\eta=1$. We then assume that the thesis holds for any monomial $x^{\eta'}$ such that $x^{\eta'}<_{\tt{Lex}}x^\eta$.

If $\vert\eta\vert> 1$, we can consider any product $x^{\eta}=x^{\eta_1}\cdot x^{\eta_2}$, $x^{\eta_1}$ and $x^{\eta_2}$ non-constant. Since $x^{\eta_i}<_{\texttt{Lex}}x^\eta, i=1,2$, we immediately obtain by induction
\[x^\eta\cdot V= x^{\eta_1}\cdot (x^{\eta_2}\cdot V)\subseteq x^{\eta_1}\langle V\rangle\subseteq\langle V\rangle.\]

If $\vert \eta\vert=1$, then we need to prove that $x_i\cdot V\subseteq\langle V\rangle$. Since $x_0V\subseteq  V$, it is then sufficient to prove the thesis for $x^\eta=x_i$, assuming that the thesis holds for every $x^{\eta'}<_{\tt{Lex}}x_i$. We consider $g_\beta=x^\delta f_\alpha \in V$, where $\max(x^\delta)\leq \min(x^\alpha)$. If $x_ig_\beta$ does not belong to $V$, then $\max(x_i\cdot x^\delta)>\min(x^\alpha)$, so $x_i>\min(x^\alpha)$. In particular, $x_i>\min(x^\alpha)\geq \max(x^\delta)$, so $x_i>_\texttt{Lex}x^\delta$: by induction, it is now sufficient to prove the thesis for $x_i f_\alpha$.

We consider an $S$-polynomial $S(f_\alpha,f_{\alpha'})=x_i f_\alpha-x^\gamma f_{\alpha'}$ such that $x^\gamma<_\texttt{Lex} x_i$. Such $S$-polynomial always exists: for instance, we can consider $x_ix^\alpha=x^{\alpha'}\ast_J x^{\eta'}$.
By the hypothesis $x_i f_\alpha-x^{\eta'}f_{\alpha'} =\sum a_j x^{{\eta'}_j}f_{\alpha_j}$ where $x^{\eta'}f_{\alpha'},  x^{{\eta'}_j}f_{\alpha_j}\in  V_\ell$  and then $x_i f_\alpha$ belongs to $\langle V\rangle$.
\end{proof}

For any strongly stable ideal $J$, with monomial basis $B_{J}=\{x^{\alpha_1},\dots,x^{\alpha_r}\}$, we can consider the set of syzygies of the following kind
\begin{equation*}
x_je_{\alpha_i}-x^\eta e_{\alpha_k}, \ \text{ with } x_j>\min(x^{\alpha_i}) \ \text{ and } \ x_jx^{\alpha_i}=x^{\alpha_k}\ast_{J}x^\eta.
\end{equation*}
This set of syzygies is actually a minimal set of generators for the first module of syzygies of $J$; this is due to Eliahou and Kervaire (see \cite{EK} and \cite[Theorem 1.31]{Gr}).

\begin{definition}
We call \emph{Eliahou-Kervaire couple} of the $J$-marked set $G$ any couple of polynomials $f_\alpha, f_\beta$, $\Ht(f_\alpha)=x^\alpha$, $\Ht(f_\beta)=x^\beta$, such that
\[
x_jx^\alpha=x^\beta\ast_Jx^\eta \text{ for some } x_j>\min(x^\alpha).
\]
We call \emph{Eliahou-Kervaire $S$-polynomial  } (EK-polynomial, for short) of $G$ an $S$-polynomial among an Eliahou-Kervaire couple of polynomials $f_\alpha$ and $f_\beta$. We denote such $S$-polynomial by $S^{EK}(f_\alpha,f_\beta)$. Observe that, thanks to the definition, an EK-polynomial is of kind
\[
S^{EK}(f_\alpha,f_\beta)=x_jf_\alpha-x^\eta f_\beta, \text{ for some } x_j>\min(x^\alpha)\text{, with } x_jx^\alpha=x^\beta\ast_Jx^\eta.
\]
\end{definition}

In the proof of Theorem \ref{BuchCrit1}, it is sufficient to assume that (\ref{buchcrit1-iii}) holds only for EK-polynomials, as stated in the following result.

\begin{corollary}\label{sizEK}
With the same notation of Theorem \ref{BuchCrit1},
\[
I\in \Mf(J)\Leftrightarrow \text{ for every EK-polynomial between elements of }G,\ S^{EK}(f_\alpha,f_\beta)\xrightarrow{\ V_\ell}0.
\]
\end{corollary}

\begin{proof}
In the proof of Theorem \ref{BuchCrit1} the crucial point is the existence of an $S$-polynomial of kind $x_if_\alpha-x^\eta f_\beta$ with $x^\eta<_{\tt{Lex}}x_i$, and we used exactly an EK-polynomial.
\end{proof}

\subsection{\texorpdfstring{The scheme structure of $\Mf(J)$}{The scheme structure of Mf(J)}}\label{subsec:MfJ_V}

Now, we recall and develop some features of the affine scheme structure of $\Mf(J)$. Let $p(t)$ the Hilbert polynomial of $S/J$ and $r$ its Gotzmann number.
In the following we will denote by $\mathcal G$ the $J$-marked set:
\begin{equation}\label{JbaseC} \mathcal{G}= \left\{F_\alpha=x^\alpha-\sum C_{\alpha\gamma} x^\gamma : \Ht(F_\alpha)=x^\alpha\in B_J,\  x^\gamma \in \cN(J)_{\vert\alpha\vert}\right\}\end{equation}
and by $\mathfrak{I}_J$ the ideal  generated by $\mathcal G$ in the ring  $K[C,x]$,  where $C$ is a compact notation for the set of new variables $C_{\alpha \gamma}$.

For every polynomial $H\in K[C,x]$, we denote by $\Supp_x(H)$ the set of monomials in the variables $x_i$ that appear in $H$ with non-null coefficients and by $\Coeff_x(H)\subset K[C]$ the set of such coefficients, that we call $x$-coefficients.

Let $\mathcal V_\ell$ and $\mathcal W_\ell$ be the analogous for $\mathcal G$ of $V_\ell$ and of $W_\ell$, respectively, for any $J$-marked set $G$. We will denote by $\mathfrak A_J$ the ideal of $K[C]$ generated by the $x$-coefficients of the $J$-reduced forms, obtained by ${\xrightarrow{\ \mathcal V_{\ell}}\ }$, of the $S$-polynomials $S(F_\alpha, F_{\alpha'})$ among elements of $\mathcal G$. This ideal does not depend on ${\xrightarrow{\ \mathcal V_{\ell}}\ }$ and defines the subscheme structure of $\Mf(J)$ in the affine space $\mathbb A^{\vert C\vert}$ (see \cite[Theorem 4.1]{CR}). Let $\mathfrak A_J^{EK}$ be the ideal of $K[C]$ generated by the $x$-coefficients of the $J$-reduced forms of the EK-polynomials of $\mathcal G$ obtained by ${\xrightarrow{\ \mathcal V_{\ell}}\ }$.

It is clear that $\mathfrak A_J^{EK}\subseteq\mathfrak A_J$. Anyway, we will prove that $\mathfrak A_J^{EK}$ and $\mathfrak A_J$ are the same ideal, although $\mathfrak A_J$ is defined by a set of generators bigger than the set of generators of $\mathfrak A_J^{EK}$. More precisely, we prove that the ideal $\mathfrak {A}_J^{EK}$ contains the $x$-coefficients of every $J$-reduced polynomial in $\mathfrak{I}_J$.

\begin{lemma}\label{formule}\
\begin{enumerate}[(i)]
\item \label{formule_i}For every monomial $x^\beta = x^\alpha \ast_J x^\delta \in J$,  there is  a formula of type
\[
x^\beta=\sum a_ix^{\gamma_i} F_{\alpha_i}+H_\beta,
\]
with $a_i\in K[C]$, $x^{\gamma_i} F_{\alpha_i}\in \mathcal V$,  $x^{\gamma_i}\leq_{Lex}x^\beta$ and $\Supp_x(H_\beta)\subset \cN(J)$.
\item \label{formule_ii}For every polynomial $x_iF_\alpha\in \mathcal W\setminus\mathcal V$,  there is  a formula of type \[
x_iF_\alpha=\sum a_jx^{\eta_j} F_{\alpha_j}+H_{i,\alpha},
\]
with $a_j\in K[C]$, $x^{\eta_j} F_{\alpha_j}\in \mathcal V$, $x^{\eta_j}<_{\mathtt{Lex}} x_i$, $\Supp_x(H_{i,\alpha})\subset \cN(J)$ and $\Coeff_x(H_{i,\alpha})\subset  \mathfrak A_J^{EK}$.
\end{enumerate}
\end{lemma}

\begin{proof}
Statement (\ref{formule_i}) follows from the existence of $J$-reduced forms obtained by ${\xrightarrow{\ \mathcal V_{\ell}}\ }$ and by Lemma \ref{reduction2}, (\ref{reduction2-iii}). Statement (\ref{formule_ii}) follows also from the definition of $\mathfrak A_J^{EK}$.
\end{proof}

\begin{proposition} \label{prop:formula}
For every polynomial $x^\delta F_\alpha\in \mathcal W\setminus \mathcal V$, we have
\begin{equation} \label{eq:formula}
x^\delta F_\alpha=\sum b_jx^{\eta_j} F_{\alpha_j}+H_{\delta,\alpha},
\end{equation}
with $b_j\in K[C]$, $x^{\eta_j} F_{\alpha_j}\in \mathcal V$, $x^{\eta_j}<_{\mathtt{Lex}} x^\delta$, $\Supp_x(H_{\delta,\alpha})\subset \cN(J)$ and $\Coeff_x(H_{\delta,\alpha})\subset  \mathfrak A_J^{EK}$.
\end{proposition}

\begin{proof}
For $\vert\delta\vert=1$ it is enough to use Lemma \ref{formule} (\ref{formule_ii}). Assume that $\vert\delta\vert>1$ and that the thesis holds for every $x^{\delta'} <_{\mathtt{Lex}} x^{\delta}$. Let $x_i=\min(x^\delta)$ and $x^{\delta'}=\frac{x^\delta}{x_i}$, so that $x^{\delta'} F_\alpha$ belongs to $\mathcal W\setminus \mathcal V$.

By the inductive hypothesis, we have $x^{\delta'}F_\alpha=\sum b'_jx^{\eta'_j} F_{\alpha_j}+H_{\delta',\alpha}$, with $x^{\eta'_j} <_{\mathtt{Lex}}x^{\delta'}$. So, multiplying by $x_i$, we obtain $x^{\delta}F_\alpha=\sum b'_jx_i x^{\eta'_j} F_{\alpha_j}+x_i H_{\delta',\alpha}$ and the thesis holds for every polynomial $x_i x^{\eta'_j} F_{\alpha_j}$ that belongs to $\mathcal W\setminus \mathcal V$ because $x_ix^{\eta'_j} <_{\mathtt{Lex}} x_ix^{\delta'}=x^\delta$. Then, we replace such polynomials by formulas of type \eqref{eq:formula} and obtain 
\[
x^{\delta}F_\alpha=\sum b_s x^{\eta_s} F_{\alpha_s}+H'+x_i H_{\delta',\alpha}
\]
where the first sum satisfies the conditions of \eqref{eq:formula} and $H'$ is $J$-reduced with $\Supp_x(H')\subset \cN(J)$ and $\Coeff_x(H')\subset\mathfrak A_J^{EK}$.

Note that $\Coeff_x(x_i H_{\delta',\alpha})=\Coeff_x(H_{\delta',\alpha}) \subset \mathfrak A_J^{EK}$, but we do not know if $\Supp_x(x_i H_{\delta',\alpha}) \subset \cN(J)$. If $x^{\beta'}\in \Supp_x(H_{\delta',\alpha})$ has $x$-coefficient $b$ in $H_{\delta',\alpha}$ and $x^\beta=x_i x^{\beta'}$ belongs to $J$, then we can use Lemma \ref{formule} (\ref{formule_i}) obtaining $b x^\beta=\sum b a_k x^{\gamma_i} F_{\alpha_k}+ b H_\beta$. Moreover, if $x^\beta=x^{\alpha'}\ast_J x^\epsilon$, then $x^{\gamma_i}\leq_{\mathtt{Lex}} x^\epsilon <_{\mathtt{Lex}} x_i <_{\mathtt{Lex}} x^\delta$, where the second inequality is due to the fact that $x^{\beta'}\in \mathcal N(J)$ and to Lemma \ref{descLex}, and all the $x$-coefficients of $H_\beta$ belong to $\mathfrak A_J^{EK}$ because they are divisible by $b$. Replacing all such monomials $x^\beta$, we obtain the thesis and $H_{\delta,\alpha}$ is $J$-reduced with $x$-coefficients in $\mathfrak A_J^{EK}$, because it is the sum of $J$-reduced polynomials with $x$-coefficients in $\mathfrak A_J^{EK}$.
\end{proof}

\begin{corollary} \label{cor:formula}
Every polynomial of $\mathfrak I_J$ can be written in a unique way as $\sum b_j x^{\eta_j} F_{\alpha_j}+H$, with $b_j\in K[C]$, $x^{\eta_j} F_{\alpha_j}\in \mathcal V$ and $H$ $J$-reduced. Moreover, we obtain also that  $\Coeff_x(H) \subset \mathfrak A_J^{EK}$.
\end{corollary}

\begin{proof}
By definition, every polynomial of $\mathfrak I_J$ is a linear combination of polynomials of $\mathcal V \cup (\mathcal W\setminus \mathcal V)$ with $x$-coefficients in $K[C]$ and, by Proposition \ref{prop:formula}, every such polynomial can be written has described in the statement. Hence, we have only to prove the uniqueness of this writing.

Let $\sum b_jx^{\eta_j} F_{\alpha_j}+H=0$ be the difference between two writings of the same polynomial of $\mathfrak I_J$, with $b_j\neq 0$, $x^{\eta_j} F_{\alpha_j}\in \mathcal V$ pairwise different and $H$ $J$-reduced. Let $x^{\eta_1} x^{\alpha_1}$ the maximum of the monomials w.r.t. the order for which $x^{\eta_i} x^{\alpha_i}$ is lower than $x^{\eta_j} x^{\alpha_j}$ if $x^{\eta_i} <_{\mathtt{Lex}} x^{\eta_j}$ or $x^{\eta_i} = x^{\eta_j}$ and $x^{\alpha_i} < x^{\alpha_j}$, where $<$ is any order fixed on $B_J$. By definition of $\mathcal V$, the unique polynomial of $\mathcal V$ with head term $x^{\eta_1} x^{\alpha_1}$ is $x^{\eta_1} F_{\alpha_1}$. Moreover, the monomial $x^{\eta_1} x^{\alpha_1}$ does not appear with a non-null coefficient in any polynomial of the sum because every other monomial belongs to $\cN(J)$ or is lower than it, by construction. Further, $x^{\eta_1} x^{\alpha_1}$ does not belong to $\Supp_x(H)$ because $\Supp_x(H)\subset \cN(J)$ and $x^{\eta_1} x^{\alpha_1}\in J$. Thus, we obtain a contradiction to the fact that $b_j\neq 0$.
\end{proof}

\begin{corollary} \label{last}
The ideal $\mathfrak A_J^{EK}$ contains the $x$-coefficients of every $J$-reduced polynomial of $\mathfrak I_J$. In particular, $\mathfrak A_J^{EK}=\mathfrak A_J$.
\end{corollary}

\begin{proof}
Let $F$ be a $J$-reduced polynomial of $\mathfrak I_J$ and let $F=\sum b_jx^{\eta_j} F_{\alpha_j}+H$ as in Corollary \ref{cor:formula}. Since $F$ itself is $J$-reduced, also $F=0+F$ is a formula as described in Corollary \ref{cor:formula} and we obtain that $F=H$, by the uniqueness of this formula. Hence, we have $\Coeff_x(F)=\Coeff_x(H) \subset\mathfrak A_J^{EK}$. The last assertion is due to the definition of $\mathfrak A_J$.
\end{proof}

\begin{remark}\label{rem5}
Actually, for every ideal $\widehat{\mathfrak A}_J\subseteq \mathfrak A_J\subseteq K[C]$ such that condition (\ref{formule_ii}) of Lemma \ref{formule} holds, also Corollary \ref{last} holds. We are then allowed to choose different sets of $S$-polynomials of $\mathcal G$ in order to obtain generators of the ideal $\mathfrak A_J$.
\end{remark}

\section{Superminimal generators and reduction}\label{secsupermin}

In this section we introduce the notion of $m$-truncation ideal  and a new polynomial reduction process, that we call {\em superminimal reduction}, useful to find a new set of equations to define a marked scheme. From the next section on, we will focus on $J$-marked schemes with $J$ a strongly stable $m$-truncation. The reason is twofold: on the one hand, strongly stable $m$-truncation ideals have a good behavior also from the geometric point of view (Theorem \ref{esempio} and Example \ref{no one-to-one}); on the other hand, the superminimal reduction is Noetherian when we take a strongly stable $m$-truncation ideal (Theorem \ref{ridsm}), but it is not if we just consider a strongly stable ideal (Example \ref{no ssr}).

\subsection{Truncation strongly stable ideals}\label{subsec31}

\begin{definition} Let $J\subseteq S$ be a {monomial idea}l. We will say that $J$ is an \emph{$m$-truncation} if $J$ is the truncation of $J^\sat$ in degree $m$, that is  $J=(J^\sat)_{\geq m}$.
\end{definition}

We observe that an $m$-truncation ideal $J$ is strongly stable if and only if $J^\sat$ is and that if $J$ is strongly stable, then $J^\sat=\underline J$.

The following Lemma highlights some simple features of $m$-truncation strongly stable ideals, which will turn out to be crucial in the proofs of our main results.

\begin{lemma}\label{proptagliati}  Let $J$ be a strongly stable $m$-truncation. Then:
\begin{enumerate}[(i)]
	\item \label{tagli}$B_{J}\cap B_{\underline J}=({B_{\underline J}})_{\geq m} $.
	\item \label{taglii} $\forall x^\beta \in B_{\underline J}\setminus B_{J}$: \ $x^\beta x_0^{m-\vert \beta \vert} \in B_{J} $.
	\item \label{tagliii}$\forall x^\gamma \in S_{\geq m}$, $\forall t\in \mathbb N$ : \ $x^\gamma x_0^t \in J \Leftrightarrow x^\gamma \in J$.
	\item \label{tagliv}$\cN(J)_{\geq m}=\cN(\underline J)_{\geq m}$.
	\item  \label{taglv}$\forall \ h \in S_{\geq m}$: \ $h$ is $J$-reduced $\Leftrightarrow$  $h$ is $\underline J$-reduced.
	\item \label{taglvi} If $I$ belongs to $\Mf(J)$, then for every homogeneous polynomial $h$ of degree $\geq m$, $J$-normal forms modulo $I$ satisfy: $\Nf(x_0^{t}\cdot h)= x_0^{t} \cdot \Nf(h)$.
\end{enumerate}
\end{lemma}

\begin{proof} Facts (\ref{tagli}) and (\ref{taglii}) are straightforward consequences of the definition of $m$-truncations. 

For (\ref{tagliii}), we only prove the non trivial part \lq\lq$\Rightarrow$\rq\rq. If $x^\gamma x_0^t \in J$, then $x^\gamma $ belongs to $\underline J$. Since $J$ is an $m$-truncation and $x^\gamma\in S_{\geq m}$, then $x^\gamma \in J$ too. 

Statements (\ref{tagliv}) and (\ref{taglv}) are obviously equivalent to (\ref{tagliii}).

For (\ref{taglvi}), we recall that   the $J$-reduced form modulo $I$ of any polynomial is  unique since  $I$ belongs to  $\Mf(J)$. By (\ref{tagliii}), both $\Nf(x_0^{t}\cdot h)$ and $  x_0^{t} \cdot \Nf(h)$ are $J$-reduced forms of $  x_0^{t} h$ and then they coincide.
\end{proof}

\begin{theorem}\label{esempio}
Let $J$ be a strongly stable $m$-truncation ideal. Two different ideals $\mathfrak a$ and $\mathfrak b$ of $\Mf(J)$ give rise to different subschemes of $\PP^n$, thus they correspond to different points of  the Hilbert scheme $\hilbp$ with $p(t)$ the Hilbert polynomial of $S/J$.
\end{theorem}

\begin{proof}

By the uniqueness of the reduced form, there is a monomial $x^\alpha\in B_{J}$ such that the corresponding polynomials $f_\alpha^{\mathfrak a}$ and $f_\alpha^{\mathfrak b}$ of the $J$-marked bases of $\mathfrak a$ and $\mathfrak b$, respectively, are different and moreover such that $f_\alpha^{\mathfrak a}\not\in \mathfrak b$ and $f_\alpha^{\mathfrak b}\not\in {\mathfrak a}$.
If $\mathfrak a$ and $\mathfrak b$ defined the same projective scheme, we would have $\mathfrak a_r=\mathfrak b_r$ for some $r\gg 0$. Hence $x_0^{r-m} f_\alpha^{\mathfrak a} - x_0^{r-m} f_\beta^{\mathfrak b}=x_0^{r-m} (-T(f_\alpha^{\mathfrak a}) +T( f_\beta^{\mathfrak b}))$   is a non-zero polynomial that belongs (for instance) to $\mathfrak a$. Moreover,  due to Lemma \ref{proptagliati}, (\ref{tagliii}), $x_0^{r-m} (-T(f_\alpha^{\mathfrak a}) +T( f_\beta^{\mathfrak b}))$   is  $J$-reduced modulo $\mathfrak a$: this is impossible  because  of Proposition \ref{cor1}, (\ref{it:cor1_iv}).
\end{proof}

The following example shows that, if $J$ is a  strongly stable ideal but not an $m$-truncation, different ideals in $\Mf(J)$ may define the same subscheme in $\PP^n$. This is the first reason why we will focus mainly on strongly stable $m$-truncations.

\begin{example}\label{no one-to-one}
 In the ring $S=K[x_0,x_1,x_2]$, let us consider the strongly stable ideal $J=(x_2,x_1^2,$ $x_1x_0)$  and for every $c\in K$ the ideal $\mathfrak a_c=(x_2+cx_1,x_1^2,x_1x_0)$. An easy computation shows that the ideals $\mathfrak a_c$ belong to $\Mf(J)$ and are pairwise different. However, $x_2^2$, $x_2x_1$, $x_2x_0$ belong to $\mathfrak a_c$: indeed $x_2^2=(x_2+cx_1)(x_2-cx_1)+c^2x_1^2, \ x_2x_1=(x_2+cx_1)x_1-cx_1^2,\ x_2x_0=(x_2+cx_1)x_0-cx_1x_0$; hence the saturation of $ \mathfrak a_c$ is $J$. Then, the subschemes $\Proj(S/\mathfrak a_c)$ of $\PP^2$ coincide. We can observe that  the difference between the ideals $\mathfrak a_c$ disappears if we only consider their homogeneous components of degree $\geq 2$.
\end{example}

\subsection{Superminimals}

In the following we will use the notation stated in Definition \ref{satmon}.

\begin{definition}\label{defsupermin}
Let $J$ be a strongly stable ideal. The \emph{set of superminimal generators} of $J$ is
\[
sB_{J}=\{x^{\beta} \in B_{J}   \ \vert\  x^{\underline{\beta}}  \in B_{\underline J}\}.
\]
\end{definition}

\begin{remark}\label{BorelGens}
Another special set of monomials for a strongly stable ideal $J$ is the so-called set of \emph{Borel generators} (see \cite{BorGen}), namely the smallest subset of $B_{J}$ such that $J$ is the minimum strongly stable ideal containing them. Although there is a clear analogy between the ideas underlying the  definition of superminimal generators and  that of  Borel generators, however they do not coincide in general.
\end{remark}

\begin{example}
Consider $\underline J:=(x_2^3,x_2^2x_1,x_2x_1^2,x_1^6)\subseteq K[x_0,x_1,x_2]$ and its $5$-truncation ideal $J := \underline{J}_{\geqslant 5}$. The set of superminimal generators of $J$ is $sB_J=\{x_2^3x_0^2,x_2^2x_1x_0^2,x_2x_1^2x_0^2,x_1^6\}$, while the set of Borel generators of $J$ is $\{x_2x_1^2x_0^2,x_1^6\}$,  because $x_2x_1^2x_0^2 \in J$ imposes $x_2^2x_1x_0^2 = \mathrm{e}^{+}_{1,2}(x_2x_1^2x_0^2) \in J$ and $x_2^3x_0^2 = \mathrm{e}^{+}_{1,2}\circ\mathrm{e}^{+}_{1,2}(x_2x_1^2x_0^2) \in J$.
\end{example}

\begin{example}
Consider $J:=(x_2^2,x_2x_1^2,x_2x_1x_0,x_2x_0^2)\subseteq K[x_0,x_1,x_2]$ whose saturation is $\underline J=(x_2)$. The set of superminimal generators of $J$ is $sB_J=\{x_2x_0^2\}$, while the set of Borel generators of $J$ is $\{x_2^2,x_2x_0^2\}$.
\end{example}

\begin{definition}\label{superminimalbasis}
Let $J$ be a strongly stable ideal. A finite set of marked polynomials $f_\beta=x^\beta-\sum c_{\beta\gamma}x^\gamma$, with $\Ht(f_\beta)=x^\beta$,  is a \emph{$J$-marked superminimal set} if the head terms form the set of superminimal generators $sB_{J}$ of $J$, they are pairwise different,  and $x^\gamma \in \cN(J)$. We call \emph{tail} of $f_\beta$ the homogeneous polynomial $T(f_\beta):=x^\beta -f_\beta$.

Every $J$-marked set $G$ contains a (unique) subset $sG$ of this type, that is called \textit{the set of superminimals of} $G$; if   $G$ is a $J$-marked basis, $sG$ is called \emph{$J$-superminimal basis}.
\end{definition}

\begin{remark}  If $\Gamma$ is a $J$-marked superminimal set,  it can always be completed to a (non-unique) $J$-marked set $G$. For instance $G=\Gamma \cup (B_{J}\setminus sB_{J})$.

On the other hand, if $I\in \Mf(J)$, then its $J$-superminimal basis is the only $J$-marked superminimal set contained in $I$. In fact, for every $x^\beta \in sB_{J}$, if $f_\beta$   belongs to both $I$ and a $J$-marked superminimal set, then $x^\beta-f_\beta$ has to be a $J$-reduced form of $x^\beta$ modulo $I$, which is the unique normal form $\Nf(x^\beta)$ modulo $I$.
\end{remark}

\begin{definition}\label{sminred} Consider  a strongly stable ideal $J$, a $J$-marked set $G$ {{and two polynomials $h$ and $h_1$. We say that $h$ is in {\em $sG_\ast$-relation} with $h_1$ if there is a monomial $x^\gamma \in \Supp(h)\cap J$, $c=\Coeff(x^\gamma)$, such that $x^\gamma$ is divisible by a superminimal generator $x^\alpha$ of $J$, {\em with} $x^\gamma=x^{\underline\alpha} \ast_{\underline J} x^\eta=x^{\alpha}\cdot x^\epsilon $ and $h_1=h-c\cdot x^\epsilon f_\alpha$, that is $h_1$ is obtained by replacing the monomial $x^\gamma$ in $h$ by $x^\epsilon \cdot T(f_{\alpha})$}}.
We call \emph{superminimal reduction} the transitive closure of the above relation and denote it by $\xrightarrow{\ sG \ast \  }$. Moreover, we say that:
\begin{itemize}
\item[-] \emph{$h$ can be reduced to $h_1$} by $\SGred$ if $h\SGred h_1$ ;
\item[-]  \emph{$h$ is non-reducible w.r.t. $\SGred$} if  no step of reduction on $h$ by $\SGred$ can be performed;
\item[-] \emph{$h$ is strongly reduced} if for every $t$, $x_0^t\cdot h$ is non-reducible w.r.t. $\SGred$.
\end{itemize}
\end{definition}

\begin{remark}\label{ultimora}\

\begin{enumerate}[(i)]
\item We use the notation $\SGred$ to underline that this reduction also involves the decomposition $\ast_{\underline J}$ of Definition \ref{stellina} and it is not the usual polynomial reduction w.r.t. a set of marked polynomials $sG$. Indeed even if a  polynomial $h$ is non-reducible  w.r.t. $\SGred$, its support can contain some monomial which is multiple of a monomial in $sB_{J}$ (see Example \ref{esempirid}); hence $h$ would be reducible w.r.t. the usual reduction $\xrightarrow{\ sG\ }$.
\item A homogeneous polynomial $h$ is strongly reduced if and only if no monomial  in $\Supp(h)$ is divisible by a monomial of $B_{\underline J}$, that is $h$ is $\underline J$-reduced.  In fact, if $x^\gamma \in \Supp(h)\cap \underline J$
then $x^\gamma=x^{\underline\alpha} \ast_{\underline J} x^\eta$ and there is $t$ such that $x^\alpha=x_0^tx^{\underline\alpha}\in B_J$. Thus $x_0^th$ can be reduced by $\SGred $ using the polynomial $f_\alpha$.
\item\label{ultimora_iii}  The polynomials $x^\epsilon f_\alpha$   that we use for the reduction procedure $\SGred$ have head terms pairwise different.
Moreover,  if $x^\delta f_{\alpha'}$  is used in the $\SGred$ reduction of   $x^\epsilon T(f_\alpha)$ then $x^{\underline \delta} <_{\texttt Lex} x^{\underline \epsilon} $.
\end{enumerate}
\end{remark}

If we consider a strongly stable ideal $J$ with no further hypothesis,  we cannot generalize the properties of the reduction $\xrightarrow{\ V_\ell}$ to $\SGred $, as shown in the following example.

\begin{example}\label{no ssr}
 In the ring $S=K[x_0,x_1,x_2]$ (with $x_2>x_1>x_0$) let us consider the  strongly stable ideal $  {J}=(x_2^3,x_2^2x_1,x_2x_1^2,x_2^2x_0,x_2x_1x_0,x_1^4,x_1^3x_0,x_1^2x_0^2)$ and its saturation $\underline J=(x_2^2,x_2x_1,x_1^2)$. The set of superminimals of $J$ is $sB_{{J}}=\{x_2^2x_0,  x_2x_1x_0,x_1^2x_0^2\}$. Let us consider the  $J$-marked superminimal set
  $sG=\{f_{x_2^2x_0}=x_2^2x_0, \ f_{x_2x_1x_0}=x_2x_1x_0-x_1^3,\ f_{x_1^2x_0^2}=x_1^2x_0^2-x_2x_0^3\}.$   The superminimal reduction  w.r.t. $sG$ is not Noetherian. For instance:
\[
 x_1^3x_0^2 {\SGred }  T(f_{x_1^2x_0^2})\cdot x_1=x_2x_0^3x_1  {\SGred } T(f_{x_2x_1x_0})\cdot x_0^2=x_1^3x_0^2.
 \]
\end{example}

However, if we assume that  the  strongly stable ideal $J$ is also an $m$-truncation ideal, then the reduction $\SGred $ turns out to be  Noetherian and satisfies several good properties, similar to the ones of $\xrightarrow{\ V_\ell\ }$.

\begin{theorem}\label{ridsm} Let $J$ be a strongly stable $m$-truncation ideal and $sG$ be a $J$-marked superminimal set. Then:
\begin{enumerate}[(i)]
\item \label{ridsm-i}$\SGred $ is Noetherian.
\item \label{ridsm-ii} For every homogeneous polynomial $h$ there exist $t$
and a unique  polynomial $ h(t)$ strongly reduced such that  $x_0^t\cdot h\SGred   h(t)$. If  $\overline{t}$ is the minimum one and  $\bar h:=h(\overline{t})$, then  $ h(t)=x_0^{t-\bar t}\cdot \bar h$ for every  $t\geq \bar t$. There is an effective procedure that computes $\bar t$ and $\bar h$.
\end{enumerate}
 If moreover $sG$ is the superminimal basis of an ideal $I$ of $\Mf(J)$, then:
\begin{enumerate}[(i)]\setcounter{enumi}{2}
\item \label{ridsm-iii} $\SGred $ computes the $J$-normal forms modulo $I$. More precisely, for every homogeneous polynomial $h$:
\[
\Nf(h)=\begin{cases}
h, &\text{ if } \deg(h) < m\\
\bar h / x_0^{\bar t}, & \text{ if } \deg(h) \geq m  \text{ and } x_0^{\overline{t}}\cdot h\SGred  {\bar h}
\end{cases}
\]
\item  \label{ridsm-iv}$\SGred$ solves the ideal-membership problem for $I$: for every homogeneous polynomial $h$:
\[ h\in I \Leftrightarrow \deg(h)\geq m \text{ and } x_0^{\overline{t}}\cdot h\SGred  0\]
\item \label{ridsm-v}There is a one-to-one correspondence between ideals in $\Mf(J)$ and $J$-superminimal bases.
\end{enumerate}
\end{theorem}

\begin{proof}\

\begin{enumerate}[(i)]
\item Since $J$ is a strongly stable and $m$-truncation ideal, then $\cN(J)_{\geq m} = \cN(\underline J)_{\geq m}$ (Lemma \ref{proptagliati}, (\ref{tagliv})). If  $\SGred $ was not Noetherian, by Lemma \ref{descLex} applied to $\underline J$, we would be able to find infinite descending chains of monomials w.r.t. $<_{\mathtt{Lex}}$.

\item\label{tagldimii} It is sufficient to prove the thesis for monomials $x^\gamma$ in $J$. Let $x^\gamma=x^{\underline\alpha}\ast_{\underline J} x^\eta$. If $x^\eta=1$, then $x^{\alpha}= x_0^{t_\alpha}\cdot x^{\underline\alpha}$ is in $sB_{J}$, $f_{\alpha}$ belongs to $sG$ and $x_0^{t_\alpha} \cdot x^{\underline\alpha} \SGred T(f_{\alpha})$, where $\Supp(T(f_{\alpha}))\subseteq \cN(J)$. In this case $\overline h=T(f_{\alpha})$ and $\overline t=t_\alpha$. If $x^\eta\neq 1$, we can assume that  the thesis holds for any monomial $x^{\gamma'}=x^{\underline\beta}\ast_{\underline J} x^{\eta'}$, such that $x^{\eta'}<_{\mathtt{Lex}} x^{\eta}$.

We perform a first reduction $x_0^{t_{\alpha}} \cdot x^\gamma \SGred   x^\eta\cdot T(f_\alpha)$. If $x^\eta\cdot T(f_\alpha)$ is strongly reduced, we are done. Otherwise, we have $x^\eta \neq x_0^{\vert \eta \vert}$. For every monomial  $x^{\gamma'}\in \Supp(x^\eta\cdot T(f_\alpha))\cap \underline J$ we have $x^{\gamma'}=x^{\underline\beta'}\ast_{\underline J} x^{\eta'}$, with  $x^{\eta'}<_{\mathtt{Lex}}x^\eta$ by Lemma \ref{descLex}.
So, we have also $x_0^t  \cdot x^{\eta'} <_{\mathtt{Lex}} x^\eta$, for every $t$. By the inductive hypothesis we can find a suitable power $t$ of $x_0$ such that every monomial in $x_0^t\cdot x^\eta \cdot T(f_\alpha)$ can be reduced by $ \SGred $ to a strongly reduced polynomial.


 It remains to prove the uniqueness of  the strongly reduced polynomial $h(t)$. Let us consider two different  strongly reduced $ \SGred$ reductions  of $x_0^t h$:   their difference  is again strongly reduced and can be written as  $\Sigma a_i x^{\eta_i}f_{\alpha_i}$ with  $a_i\in K, \ a_i\neq 0$ and $x^{\eta_i}f_{\alpha_i}$ pairwise different. Let  $ x^{\eta_1}f_{\alpha_1}$ be such that  for every $i\geq 2$, either  $x^{\eta_1}>  _{\texttt Lex}x^{\eta_i}$ or $x^{\eta_1} = x^{\eta_i}$ and  $x^{\alpha_1}>_{\texttt Lex}x^{\alpha_i}$. Then   $ x^{\eta_1}x^{\alpha_1}$ should cancel with a monomial in $\Supp(x^{\eta_i}T(f_{\alpha_i}))$ for some $i$, but this is impossible as observed in Remark \ref{ultimora}, (\ref{ultimora_iii}).

Observe that, though for a fixed $x^{\gamma}=x^{\underline\alpha}\ast_{J} x^{\eta}$,  there are infinitely many monomials $x^{\gamma'}=x^{\underline\beta}\ast_{J} x^{\eta'}$ such that $x^{\eta'}<_{\mathtt{Lex}} x^{\eta}$, we use the inductive hypothesis  only with respect to the finite number of them that  appear on the support of  $\cdot x^\eta \cdot T(f_\alpha)$. For this reason our procedure is effective.
\end{enumerate}

From now on we consider $I\in \Mf(J)$;  therefore   if $h$ is a homogeneous polynomial and $ h\SGred  h_1$ with $h_1$  strongly reduced,  then by uniqueness of $J$-normal forms modulo $I$ we have $h_1=\Nf(h)$.
\begin{enumerate}[(i)]\setcounter{enumi}{2}
\item If $\deg h < m$ we are done. Otherwise from (\ref{tagldimii}) we have that $x_0^{\overline{t}}\cdot h\SGred  {\overline{h}}$  and $\bar h$ is a $J$-reduced form modulo $I$.  Thus  $x_0^{t} \cdot \Nf(h)$ is $J$-reduced  too (Lemma \ref{proptagliati}, (\ref{tagliii})) and  we get the desired equality by uniqueness of $J$-normal forms modulo $I$.

\item   This is a consequence of (\ref{ridsm-iii}) and of Proposition \ref{cor1} (\ref{it:cor1_iv}).

\item  This is the straightforward consequence of (\ref{ridsm-iv}).\qedhere
\end{enumerate}
\end{proof}

Whenever $J$ is a strongly stable $m$-truncation ideal and $sG$ is the superminimal basis of an ideal $I\in \Mf(J)$, then     $sG$ is a subset of the set $V$ of Definition \ref{def:$V_m$}. Nevertheless,
it is interesting to notice that not every step of reduction by $\SGred $ is also a step of reduction by $\xrightarrow{\ V_\ell\ }$, as shown in the following example.

\begin{example} \label{esempirid}
Consider $J=(x_1^2,x_0x_2,x_1x_2,x_2^2)$ which is a strongly stable ideal  and a $2$-truncation of $\underline J=(x_2,x_1^2)$ in $K[x_0,x_1,x_2]$. Let $G$ be a $J$-marked set.
\begin{itemize}
  \item The monomial $x_2\cdot x_1^2$ is non-reducible w.r.t. $sG$, because the only monomial of $sB_{J}$ dividing it is $x_1^2$, but $x_2x_1^2=x_2\ast_{\underline J}x_1^2$. On the other hand, $x_2x_1^2=x_2x_1\ast_J x_1$, so $x_2x_1^2\xrightarrow{\ V_3\ }x_1 T(f)$ where $f\in V_2$, $\Ht(f)=x_2x_1$.
 \item The only way to reduce $x_0\cdot x_2^2$ via $\xrightarrow{\ V_3\ }$ leads to $x_0\cdot T(f')$, where $f'$ is the unique polynomial of $V_2$ such that $\Ht(f')=x_2^2$. Moreover, $x_0\cdot T(f')$ is not further reducible, because all the monomials of its support belong to $\cN(J)$. On the other hand, according to Definition \ref{sminred}, a first step of reduction of the monomial $x_0\cdot x_2^2$ via  $\SGred $ is $x_0 x_2^2\SGred x_2\cdot T(f'')$, where $f''$ is the polynomial in $sG$ with $\Ht(f'')=x_0\cdot x_2$. Since $x_2$ is a monomial of $B_{\underline J}$, every monomial appearing in $\Supp(x_2\cdot T(f''))$ belongs to $J$, and so we will need further steps of reduction via $\SGred $ to compute a polynomial non-reducible w.r.t. $sG$.
 \end{itemize}
\end{example}




\section{Buchberger-like criterion by superminimal reduction}\label{buchsec}

In the present and following sections,  we assume that {\em $J\subseteq S$ is a strongly stable $m$-truncation ideal}, in order to apply the main results of Section \ref{secsupermin}, mainly those concerning the new reduction process $\SGred $ (Theorem \ref{ridsm}).
We will also use the sets of polynomials $V$ and $W$  which are defined from a $J$-marked set $G$ (see Definition \ref{def:$V_m$}), and the reduction relation $\xrightarrow{\ V_\ell}$.

In Section \ref{buch}, we proved that $J$-marked bases are characterized by a Buchber\-ger-like criterion on the reduction of $S$-polynomials between elements of $G$ by $\xrightarrow{\ V_\ell}$ (Theorem \ref{BuchCrit1}). Afterwards, in Section \ref{secsupermin} we showed that every homogeneous ideal $I$ in $\Mf(J)$ is completely determined by its superminimal basis $sG$ and that $J$-normal forms modulo $I$ can be computed using $\SGred $, that is again  using  polynomials in the subset $sG$ of $G$ (Theorem \ref{ridsm}).

Therefore,  it is natural to ask whether one can obtain a Buchberger-like criterion only considering $S$-polynomials among elements in $sG$. Unfortunately, the answer is negative, as clearly shown by Example \ref{nonbastano}. However, we can prove a few variants of the Buchberger-like criterion of Theorem \ref{BuchCrit1}, in which the set of superminimals $sG$ and the superminimal reduction process $\SGred $  replace $G$ and $\xrightarrow{\ V_\ell}$. Using these new criteria in the next section we will be able to obtain sets of equations defining $\Mf(J)$ in a smaller set of variables than those of Subsection \ref{subsec:MfJ_V}.

\begin{example}\label{nonbastano}
We consider the strongly stable $2$-truncation ideal 
\[
J=(x_3^2,x_3x_2,x_3x_1,x_3x_0,x_2^2)\subseteq K[x_0,x_1,x_2,x_3]
\]
 whose saturation is $\underline J=(x_3,x_2^2)$. In this case, $sB_{J}$ contains only two monomials, $x_3x_0$ and $x_2^2$. If $G$ is any $J$-marked set, then $sG=\{f_{x_3x_0},f_{x_2^2}\}$. The unique $S$-polynomial among superminimal elements is
\[
S(f_{x_3x_0},f_{x_2^2})=x_2^2f_{x_3x_0}-x_3x_0f_{x_2^2}=x_3x_0\cdot T(f_{x_2^2})-x_2^2\cdot T(f_{x_3x_0}).
\]
Any monomial appearing in $\Supp(T(f_{x_3x_0}))$ is in $\cN(J)_2=K[x_0,x_1,x_2]_2\setminus\{x_2^2\}$. Then any monomial appearing in $\Supp(x_2^2\cdot T(f_{x_3x_0}))$ is further reduced by $f_{x_2^2}$, obtaining by $\xrightarrow{\ V_4\ }$ or $\SGred $
\[
S(f_{x_3x_0},f_{x_2^2})=x_3x_0\cdot T(f_{x_2^2})-T(f_{x_2^2})\cdot T(f_{x_3x_0})=T(f_{x_2^2})\cdot f_{x_3x_0}\rightarrow 0.
\]
Nevertheless, even if the only $S$-polynomial among superminimal generators reduces to 0, if we consider $sG=\{f_{x_3x_0},f_{x_2^2}\}$ with $f_{x_3x_0}=x_3x_0+x_1^2$ and $f_{x_2^2}=x_2^2$, then for any choice of $f_{x_3^2},f_{x_3x_2},f_{x_3x_1}$, the $S$-polynomial among $f_{x_3x_1}$ and $f_{x_3x_0}$ does not reduce to 0:
\[
S(f_{x_3x_1},f_{x_3x_0})=x_0f_{x_3x_1}-x_1f_{x_3x_0}=\sum_{x^{\alpha_i}\in \cN(J)_2}a_i x^{\alpha_i}x_0-x_1^3.
\]
The monomials $x^{\alpha_i} x_0$ are in $\cN(J)_3$ and are strongly reduced. Furthermore, $x_1^3$ does not appear among monomials $x^{\alpha_i}x_0$, so it is not canceled, and it is strongly reduced too. Therefore, for any choice of coefficients in the tail of $f_{x_3x_1}$, we have an $S$-polynomial which is not reducible to 0, and any $J$-marked set containing $f_{x_3x_0}=x_3x_0+x_1^2$ is not a $J$-marked basis.
\end{example}

\subsection{\texorpdfstring{Buchberger-like criteria via $\SGred $: first variant}{Buchberger-like criteria via sG* reduction: first variant}}\label{subsec:bucsec1}

In this subsection  we prove that the Buchberger-like criterion of Theorem \ref{BuchCrit1} and  Corollary \ref{sizEK} can be rephrased in terms of the reduction process $\SGred $ . The involved $S$-polynomials will be all those between elements in $G$ (Theorem \ref{nostrobuch}), or only $EK$-polynomials between elements of $G$ (Corollary \ref{nostrobuchEK}). We will need a few lemmas.

\begin{lemma}\label{x0V}
 Let $J$ be a  strongly stable $m$-truncation ideal, $G$ be a $J$-marked set and $h$ be a homogeneous polynomial of degree $\ell\geq m$. Then:
\[
h\in \langle V_\ell\rangle \Leftrightarrow x_0\cdot h\in \langle V_{\ell+1} \rangle.
\]
\end{lemma}

\begin{proof}
If $h \in \langle V_\ell\rangle$, then $x_0\cdot h\in \langle V_{\ell+1}\rangle $ by definition of $V$.

Vice versa, assume that $x_0\cdot h \in \langle V_{\ell+1}\rangle$. This is equivalent to $x_0\cdot h\xrightarrow{\ V_{\ell+1}\ }0$. Every monomial in $\Supp(x_0\cdot h)$ can be written as $x_0\cdot x^\epsilon$; observe that $x_0\cdot x^\epsilon \notin B_{J}$, because $\deg(x_0\cdot x^\epsilon)>m$, by Lemma \ref{proptagliati}, (\ref{tagli}). Then, if $x_0\cdot x^\epsilon$ belongs to $J$, we can decompose it as $x_0\cdot x^\epsilon =x^\alpha\ast_{J} x^\eta$, $x^\alpha \in B_J$ and $x^\eta\neq 1$. Since $\min(x^\alpha)\geq \max(x^\eta)$, we have that $x^\eta$ is divisible by $x_0$. So $x^\eta=x_0\cdot x^{\eta'}$.

Summing up, in order to reduce the monomial $x_0\cdot x^\epsilon$ of $\Supp(x_0\cdot h)$ using $V$, we use the polynomial $x_0\cdot x^{\eta'}\cdot f_\alpha \in V$,  $\Ht(f_\alpha)=x^\alpha$. If the coefficient of $x_0\cdot x^\epsilon$ in $x_0\cdot h$ is $a$, we obtain
\[
x_0\cdot h\xrightarrow{\ V_{\ell+1}} x_0\cdot(h-a\cdot x^{\eta'}f_\alpha).
\]
At every step of reduction, we obtain a polynomial which is divisible by $x_0$. In particular,
\[
x_0\cdot h \in \langle V_{\ell+1}\rangle \Rightarrow x_0\cdot h=x_0\cdot \sum a_i x^{\eta_i}f_{\alpha_i},\text{ where } x_0\cdot x^{\eta_i}f_{\alpha_i}\in V_{\ell+1}.
\]
Then we have that $h= \sum a_i x^{\eta_i}f_{\alpha_i}$ and $x^{\eta_i}f_{\alpha_i}\in V_{\ell}$, that is $h\in \langle V_{\ell}\rangle$.
\end{proof}

Consider $f_\alpha,f_{\alpha'}\in G$, the $S$-polynomial $S(f_\alpha, f_{\alpha'})=x^\gamma f_\alpha-x^{\gamma'}f_{\alpha'}$ and assume that $x^{\gamma'}<_{\mathtt{Lex}}x^{\gamma}$. By Lemma \ref{reduction2}, (\ref{reduction2-iii}), if $S(f_\alpha,f_\alpha')\xrightarrow{\ V_\ell} h$, then $S(f_\alpha,f_\alpha')-h=\sum a_jx^{\delta_j}f_{\beta_j}$
with $x^{\delta_j}f_{\beta_j} \in V_\ell$, $x^{\delta_j}<_{\mathtt{Lex}}x^\gamma$.
{{Now we show that a similar result holds for the superminimal reduction $\SGred $}}.

\begin{lemma}\label{lexrid}
Let  $J$ be a  strongly stable $m$-truncation ideal, $G$ be a $J$-marked set and  $f_\alpha,f_{\alpha'}$ be two polynomials belonging to $G$. Consider the $S$-polynomial $S(f_\alpha, f_{\alpha'})=x^\gamma f_\alpha-x^{\gamma'}f_{\alpha'}$, with $x^{\gamma'}<_{\mathtt{Lex}}x^{\gamma}$. If $x_0^t\cdot S(f_\alpha,f_\alpha')\SGred  h$, then $x_0^t\cdot S(f_\alpha,f_\alpha')-h=\sum a_jx^{\eta_j}f_{\beta_j}$
with $f_{\beta_j} \in sG$, $x^{\eta_j}<_{\mathtt{Lex}}x^\gamma$ and $x^{\underline{\eta_j}}<_{\mathtt{Lex}}x^{\underline\gamma}$.
\end{lemma}

\begin{proof}
Every  monomial $x_0^t\cdot x^\gamma\cdot x^\epsilon$ in $\Supp(x_0^t\cdot x^\gamma\cdot T(f_\alpha))\cap J$ decomposes as $x_0^t\cdot x^\gamma\cdot x^\epsilon=x^{\underline\beta}\ast_{\underline J} x^{\eta}$, $x^{\eta}<_{\mathtt{Lex}}x_0^t x^\gamma$ and $x^{\underline\eta}<_{\mathtt{Lex}}x^{\underline\gamma}$ by Lemma \ref{proptagliati}, (\ref{tagliv}) and Lemma \ref{descLex}. The same holds for any further reduction and the same argument applies to monomials appearing in $\Supp(x_0^t\cdot x^{\gamma'}\cdot T(f_{\alpha'}))$.
\end{proof}

We point out that Lemma \ref{lexrid} does not hold without the hypothesis that $J$ is an $m$-truncation ideal, as shown by the following example.

\begin{example}\label{jnontagl}
In $S=K[x_0,x_1,x_2,x_3]$, consider the strongly stable ideal
\[
J=(x_3^2,x_3x_2,x_3x_1)_{\geq 4}+(x_2^2)_{\geq 6},
\]
whose saturation is $\underline J=(x_3^2,x_2x_3,x_1x_3,x_2^2)$. $J$ is not an $m$-truncation for any $m$. Consider a $J$-marked set $G$ and $f_\alpha,f_\beta \in G$ such that $\Ht(f_\alpha)=x_0^2x_2x_3$ and $\Ht(f_\beta)=x_0^2x_1x_3$ and consider $x_2^4\in \Supp(T(f_\beta))$. Then $S(f_\alpha,f_\beta)=x_1f_\alpha-x_2f_\beta$. If we apply Definition \ref{sminred}, we reduce $x_2^4\in \Supp(S(f_\alpha,f_\beta))$ by $\SGred $, pre-multiplying by $x_0^4$. We get that  $x_0^4x_2^4$ belongs to $\Supp(x_0^4S(f_\alpha,f_\beta))$ and $x_0^4x_2^4=x_2^2\ast_{\underline J}x_0^4x_2^2$. But $x_2^2>_{\tt{Lex}} x_2$.
\end{example}

\begin{theorem}\label{nostrobuch}
Let $J$ be a strongly stable $m$-truncation ideal, $G$ be a $J$-marked set and $I$ be the homogeneous ideal generated by $G$. The followings are equivalent:
\begin{enumerate}[(i)]
\item \label{nostrobuch_i}$I\in \Mf(J)$;
\item\label{nostrobuch_ii} $ \forall f_\alpha, f_{\alpha'} \in G,\ \exists \ t$ such that  $x_0^t\cdot S(f_\alpha, f_{\alpha'}) \SGred  0$;
\item \label{nostrobuch_iii}$ \forall f_\alpha, f_{\alpha'} \in G,\ \exists\ t$ such that  $x_0^t\cdot S(f_\alpha, f_{\alpha'})=x_0^t(x^\gamma f_\alpha-x^{\gamma'}f_{\alpha'}) =\sum a_j x^{\eta_j}f_{\alpha_j}$, with $x^{\eta_j}<_{\tt{Lex}}\max_{\tt{Lex}}\{x^\gamma,x^{\gamma'}\}$ and $f_{\alpha_j}\in sG$.
\end{enumerate}
\end{theorem}

\begin{proof}
If $I\in \Mf(J)$, we can apply Theorem \ref{ridsm}, (\ref{ridsm-iv}) because any $S$-polynomial among elements in $G$ belongs to $I$.

If statement (\ref{nostrobuch_ii}) holds, then we get (\ref{nostrobuch_iii}) by Lemma \ref{lexrid}.

We now assume that statement (\ref{nostrobuch_iii}) holds and by Lemma \ref{forseteo} it is sufficient to prove that $\langle V\rangle=\langle W\rangle$ using an argument analogous to that applied in the proof of Theorem \ref{BuchCrit1}. It is sufficient to prove that $x^\eta\cdot V\subseteq \langle V\rangle$, for every monomial $x^\eta$. We proceed by induction on the monomials $x^\eta$, ordered according to $>_{\tt{Lex}}$. The thesis is obviously true for $x^\eta=1$. We then assume that the thesis holds for any monomial $x^{\eta'}$ such that $x^{\eta'}<_{\tt{Lex}}x^\eta$.

If $\vert\eta\vert> 1$, we can consider any product $x^{\eta}=x^{\eta_1}\cdot x^{\eta_2}$, $x^{\eta_1}$ and $x^{\eta_2}$ non-constant. Since $x^{\eta_i}<_{\texttt{Lex}}x^\eta, i=1,2$, we immediately obtain by induction
\[x^\eta\cdot V = x^{\eta_1}\cdot (x^{\eta_2}\cdot V)\subseteq x^{\eta_1}\langle V\rangle\subseteq\langle V\rangle.\]

If $\vert \eta\vert=1$, then we need to prove that $x_i\cdot V\subseteq\langle V\rangle$. Since $x_0 V\subseteq  V$, it is then sufficient to prove the thesis for $x^\eta=x_i$, $i\geq 1$, assuming that the thesis holds for every $x^{\eta'}<_{\tt{Lex}} x_i$. We consider $g_\beta=x^\delta f_\alpha \in V$, where $\max(x^\delta)\leq \min(x^\alpha)$. If $x_ig_\beta$ does not belong to $V$, then $\max(x_i\cdot x^\delta)>\min(x^\alpha)$, so $x_i>\min(x^\alpha)$ because $\max(x^\delta)\leq \min(x^\alpha)$ by construction. In particular, $x_i>\min(x^\alpha)\geq \max(x^\delta)$, so $x_i>_\texttt{Lex}x^\delta$ and it is sufficient to prove the thesis for $x_i f_\alpha$.

We consider an $S$-polynomial $S(f_\alpha,f_{\alpha'})=x_i f_\alpha-x^\gamma f_{\alpha'}$ such that $x^\gamma<_\texttt{Lex}x_i$. Such $S$-polynomial always exists: for instance, we can consider $x_ix^\alpha=x^{\alpha'}\ast_{J} x^{\eta'}$.

By hypothesis there is $t$ such that $x_0^t S(f_\alpha,f_{\alpha'})=x_0^t(x_i f_\alpha-x^{\eta'}f_{\alpha'}) =\sum a_j x^{{\eta'}_j}f_{\alpha_j}$ where $x_0^tx^{\eta'}, x^{{\eta'}_j}$ are lower than $x_i$ w.r.t. $\texttt{Lex}$. Then $x_0^tx^{\eta'}f_{\alpha'}$,  $x^{{\eta'}_j}f_{\alpha_j}$  belong to $\langle V\rangle$ by induction and we conclude that $x_i f_\alpha \in \langle V \rangle$, by Lemma \ref{x0V}.
\end{proof}

The previous theorem is the analogous of Theorem \ref{BuchCrit1} for the reduction process $\SGred $.
As stated in Corollary \ref{sizEK} concerning the Buchberger-like criterion for the reduction $\xrightarrow{\ V_\ell\ }$, also in Theorem \ref{nostrobuch} it would be sufficient to assume statement (\ref{nostrobuch_iii}) only for EK-polynomials.

\begin{corollary}\label{nostrobuchEK}
With the same notations of Theorem \ref{nostrobuch}, the followings are equivalent:
\begin{enumerate}[(i)]
\item $I\in \Mf(J)$
\item { for every EK-polynomial between elements of }$G$, $\exists \ t :\, x_0^tS^{EK}(f_\alpha,f_{\alpha'})\SGred 0$.
\item  { for every EK-polynomial between elements of }$G$, $\exists \ t$ such that  $x_0^t\cdot S^{EK}(f_\alpha, f_{\alpha'})=x_0^t(x_i f_\alpha-x^{\gamma'}f_{\alpha'}) =\sum a_j x^{\eta_j}f_{\alpha_j}$, with $x^{\eta_j}<_{\tt{Lex}}x_i$ and $f_{\alpha_j}\in sG$.
\end{enumerate}
\end{corollary}

\subsection{\texorpdfstring{Buchberger-like criteria via $\SGred $: second variant}{Buchberger-like criteria via sG* reduction: second variant}}\label{subsec:bucsec2}

As before, let $J$ be a strongly stable $m$-truncation ideal. By Example \ref{nonbastano}, we have already shown that reductions of $S$-polynomials between elements of $sG$ are not sufficient to characterize ideals of $\Mf(J)$; hence some more conditions are necessary. To this aim, we add some further $S$-polynomials.

Indeed, Theorem \ref{paolo} uses the set $L_1$ of some couples of polynomials of $sG$ and the set $L_2$ of some particular couples of elements of $G$ of minimal degree $m$ to obtain a new characterization of $\Mf(J)$.
Actually the elements of $L_1$ are not all the possible couples of elements in $sG$, but a subset of them, corresponding to a minimal set of generators for the first module of syzygies of the Eliahou-Kervaire resolution of $\underline J$.

\begin{theorem}\label{paolo}
Consider a strongly stable $m$-truncation ideal $J$ and $G$ a $J$-marked set. Let us define the following sets:
\[
L_1 := \left\{ (f_\alpha,f_{\alpha'})\  \vert\ f_\alpha,f_{\alpha'} \in sG  \text{ and } x_i x^{\underline \alpha} = x^{\underline \alpha'} \ast_{\underline J} x^\eta  \right\},
\]
\[
L_2:=\left\{ (f_\alpha,f_{\alpha'})\  \vert\ f_\alpha,f_{\alpha'} \in G_m  \text{ and } x_i x^{\alpha'}=x_0 x^{\alpha},\  x_i = \min_{j>0} \{x_j : x_j \mid x^\alpha\}\right\}.
\]
Then:
\[
 I\in\Mf(J)\Leftrightarrow\ \forall\  (f_\alpha,f_{\alpha'})\in L_1\cup L_2, \ \exists \ t \text{ such that }x_0^t\cdot S(f_\alpha,f_{\alpha'})\SGred 0.
\]
\end{theorem}

\begin{proof}
If $I$ belongs to $\Mf(J)$, then it is enough to apply Theorem \ref{nostrobuch}, (\ref{nostrobuch_ii}).

Vice versa, by Lemma \ref{forseteo} it is sufficient to prove that $\langle V\rangle=\langle W\rangle$, that is $x_i\cdot V\subseteq \langle V\rangle$ for every $i=0,\dots, n$. We proceed by induction on the variables. By construction we have $x_0\cdot V \subseteq \langle V\rangle$. We now assume that $(x_0,\dots,x_{i-1})V\subseteq \langle V\rangle$ and we prove that $x_i\cdot V\subseteq \langle V\rangle$. Consider $x^\delta f_\beta \in V$. The thesis is that $x_i\cdot x^\delta f_\beta$ is contained in $\langle V\rangle$. If $x_i x^\delta f_\beta$ does not belong to $V$, then $\max(x_i\cdot x^\delta)>\min(x^\beta)$, so $x_i>\min(x^\beta)$ because $\max(x^\delta)\leq \min(x^\beta)$ by construction. In particular, $x_i>\min(x^\beta)\geq \max(x^\delta)$, so that it is sufficient to prove the thesis for $x_i f_\beta$, because by induction then we have $x^\delta x_i f_\beta\in \langle V\rangle$. Consider $x^\beta=x^{\underline\alpha}\ast_{\underline J} x^\eta$.

We have a first case when $x^\eta=1$. Then $x^\beta=x^{\underline\alpha}$ and $f_\beta$ belongs to $sG$. We consider $x^{\underline\alpha} x_i=x^{\underline{\alpha'}}\ast_{\underline J}x^{\eta'}$. Observe that since $x_i>\min(x^{\underline\alpha})$ then $x_i$ does not divide $x^{\eta'}$ and $\max(x^{\eta'})<x_i$. Consider $x^{{\alpha}'}=x^{\underline{\alpha'}}\cdot x_0^{t_{\alpha'}}$, so that we can take the polynomial $f_{{\alpha}'}\in sG$. The pair $(f_\beta,f_{\alpha'})$ belongs to $L_1$, hence, by the hypothesis and by Lemma \ref{lexrid}, there is $t$ such that
\[
x_0^tS(f_\beta,f_{{\alpha}'})=x_0^t(x_0^{t_{{\alpha}'}}x_if_\beta-x^{\eta'}f_{{\alpha}'})=\sum a_j x^{\eta_j}f_{\alpha_j},
\]
with $x^{\eta_j}<_{\tt{Lex}}x_i$ and $f_{\alpha_j}\in sG$. Hence we obtain that both  $x^{\eta_j}f_{\alpha_j}$ and  $x^{\eta'}f_{{\alpha}'}$ belong to $\langle V\rangle$ by induction on the variables, and so $x_if_\beta$ belongs to $\langle V\rangle$ (by Lemma \ref{x0V}).

We have a second case when $x^\eta=x_0^t$, $t>0$. Then, $\vert\beta\vert=m$ and $f_\beta$ belongs to $sG$. Let $x_ix^\beta = x^{\underline{\alpha}'}\ast_{\underline J} x^{\eta'}$. If $x_i>\min (x^{\underline{\alpha}'})$, then $x^{\eta'}$ is not divisible by $x_i$ and we repeat the argument above. Otherwise, $x_i\leq\min (x^{\underline{\alpha}'})$ and $x_i$ does not divide $x^{\eta'}$, so that $x_i=\min(x^{\underline{\alpha}'})$ and $x^{\eta'}<_{\mathtt{Lex}} x_i$. Then, we take $x^{\beta'}=\frac{x^\beta}{x_0}\cdot x_i$ that belongs to $B_{J}$ because it has degree $m$. The pair $(f_\beta,f_{\beta'})$ belongs to $L_2$ and we repeat the same reasoning above.

We now assume the thesis holds for every $f_{\beta'}$ such that $x^{\beta'}=x^{\underline{\alpha}'} \ast_{\underline J} x^{\eta'}$ with $x^{\eta'}<_{\tt{Lex}}x^\eta$. By the base of the induction, we can suppose that $x^\eta\geq_{\tt{Lex}}x_1$; so, $f_\beta$ does not belong to $sG$ and it has degree $m$. Let $x_j:= \min_{l>0} \{x_l\ :\ x_l\mid x^\beta\}$.

Observe that if $x_0$ does not divide $x^\beta$, then $x_j= \min(x^\beta)$; in this case, we have  $x_i>x_j$ because  $x_i > \min (x^\beta)$. Anyway, first we suppose that $x_i\leq x_j$;  $x_j> \min(x^\beta)$ and $x_0$ divides $x^\beta$. We consider $x^{\beta'}=\frac{x^\beta}{x_0}\cdot x_i$, the pair $(f_{\beta},f_{\beta'})$
that belongs to $L_2$ and we repeat the argument of the previous case.

We now assume that $x_i>x_j$ and consider $x^{\beta'}=\frac{x^\beta}{x_j}\cdot x_0=x^{\underline{\alpha}'}\ast_{\underline J}x^{\eta'}$. Observe that $x^{\eta'}<_{\tt{Lex}}x^\eta$ as $x^{\eta'}=\frac{x^\eta}{x_j}\cdot x_0$. Therefore the pair $(f_{\beta'},f_\beta)$
belongs to $L_2$;  by the hypothesis and by Lemma \ref{lexrid}, there is an integer $t$ such that
\begin{equation}\label{spoly}
x_0^tS(f_{\beta'},f_\beta)=x_0^t(x_j f_{\beta'}-x_0 f_\beta)=\sum a_l x^{\eta_l}f_{\alpha_l}
\end{equation}
with $x^{\eta_l}<_{\tt{Lex}}x_j$, $f_{\alpha_l}\in sG$. We now multiply \eqref{spoly} by $x_i$. We observe that $x_i f_{\alpha_l}$ belongs to $\langle V\rangle$, because $f_{\alpha_l}\in sG$ and by the first two cases. Also $x_i f_{\beta'}$ belongs to $\langle V\rangle$ because $x^{\eta'}<_{\tt{Lex}}x^\eta <_{\tt{Lex}} x_i$. Moreover, $x_jx_i f_{\beta'}$ belongs to $\langle V\rangle$ by induction on the variables. Finally $x_i f_\beta$ belongs to $\langle V\rangle$ thanks to Lemma \ref{x0V}.
\end{proof}


\section{\texorpdfstring{Embedding of $\Mf(J)$ in affine linear spaces of low dimension}{Embedding of Mf(J) in affine linear spaces of low dimension}}\label{ctildate}

In this section we continue to consider a strongly stable $m$-truncation ideal $J=\underline{J}_{\geq m}$ and, as in Subsection \ref{subsec:MfJ_V}, we work again with $J$-marked sets $\mathcal{G}$ where the coefficients of the monomials in the tails are considered as parameters.

\begin{definition}\label{decomposizione} If  $\mathcal G$ is the set of marked polynomials given as in \eqref{JbaseC} for the ideal $J$, we will call \emph{set of superminimals}, and denote it by $s\mathcal G$,  the subset of $\mathcal G$ made up of $F_\alpha \in \mathcal G$ with $\Ht\left(F_\alpha\right)\in sB_{J}$. We will denote by $C$ the set of variables appearing in the tails of the polynomials in $\mathcal G$ and by $\widetilde C$ the set of variables appearing in the tails of the polynomials in $s\mathcal G$. ${\mathfrak A}_{J}$ is the ideal defining the affine subscheme $\Mf(J)$ in the ring $K[C]$.
\end{definition}

Observe that the $J$-marked basis $G$ of every $I\in \Mf(J)$ is obtained by specializing in a suitable way the variables $C$ in $\mathcal G$ and that the set of superminimals $sG$ of $I$ is obtained in the same way by $s\mathcal G$ through the same specialization of the variables $\widetilde C$.

\subsection{\texorpdfstring{The new embedding of $\Mf(J)$}{The new embedding of Mf(J)}}

In this subsection we answer to the first question raised in the Introduction. In Theorem \ref{smallaffine} we prove that the set of equations in $K[C]$ defining $\Mf(J)$ allows the elimination of a large number of parameters, more precisely those of $C\setminus \widetilde{C}$.  Furthermore, using results of previous sections about the superminimal reduction, we are able to determine a set of equations defining $\Mf(J)$ in $K[\widetilde{C}]$ avoiding at all the introduction of parameters in  $C\setminus \widetilde{C}$. This fact combined with the choice of a small set of $S$-polynomials (according to Corollary \ref{nostrobuchEK} or Theorem \ref{paolo}) will turn out to be significantly useful in projecting an effective algorithm for the computation of such equations.   Furthermore, this new sets of equations turns out to be more suitable  in order to compare marked schemes of $m$-truncation ideals of a  strongly stable saturated ideal $\underline J$  as $m$ varies.

\begin{definition}\label{elim}
Let $x^\alpha \in B_{J}$ and $t$ be an integer such that $x_0^t \cdot x^\alpha \SGredPar  H_\alpha \in K[\tilde C,x]$, with $H_\alpha$ strongly reduced (the integer $t$ exists by Theorem \ref{ridsm}). We can write $H_\alpha=H'_\alpha  + x_0^t \cdot H''_\alpha$, where no monomial appearing in $H'_\alpha$ is divisible by $x_0^t$. We will denote by:
\begin{itemize}
\item $\mathfrak B=\{C_{\alpha \gamma}- \phi_{\alpha \gamma} \ : \ x^\alpha \in B_{J}\setminus sB_{J}, x^\gamma \in \cN(J)_{\vert \alpha \vert}\}$ the set of the $x$-coefficients of $T(F_\alpha)-H''_\alpha$ for every $x^\alpha \in B_{J}$;
\item $\mathfrak D_1\subset K[\widetilde C]$ the set of the $x$-coefficients of $H'_\alpha$ for every $x^\alpha \in B_{J}\setminus sB_{J}$;
\item $\mathfrak D_2$ the set of the $x$-coefficients of the strongly reduced polynomials in $(s{\mathcal G})K[\widetilde C,x] $.
\end{itemize}
\end{definition}

\begin{remark}
Observe that not only $\mathfrak D_2$ but also  $\mathfrak B$ and $\mathfrak D_1$ are well-defined thanks to the uniqueness of $H_\alpha$, by Theorem \ref{ridsm}, (\ref{ridsm-ii}).
\end{remark}

\begin{theorem}\label{smallaffine} The $J$-marked scheme $\Mf(J)$ is defined by the ideal $\widetilde {\mathfrak A}_{J}:={\mathfrak A}_{J} \cap K[ \widetilde C]$ as subscheme of the affine space $\mathbb A^{\vert \widetilde C\vert}$, where $\vert \widetilde C\vert= \sum_{x^{\alpha}\in sB_{J}}\vert \cN(J)_{\vert\alpha\vert}\vert$.
Moreover ${\mathfrak A}_{J}=(\mathfrak B \cup \mathfrak D_1 \cup \mathfrak D_2)K[C]$  and $\widetilde {\mathfrak A}_{J}=(\mathfrak D_1 \cup \mathfrak D_2)K[\widetilde C]$.
\end{theorem}

\begin{proof}
For the first part it is sufficient to prove that ${\mathfrak A}_{J}$ contains $\mathfrak B$ and so it contains an element of the type $C_{\alpha \gamma}- \phi_{\alpha \gamma}$, for every $C_{\alpha \gamma}\in C\setminus \widetilde C$, where $\phi_{\alpha \gamma}\in K[\widetilde C]$, that allows the elimination of the variables $C_{\alpha \gamma}\in C\setminus \widetilde C$.

It is clear by the construction in Definition \ref{elim} that $H_\alpha$ belongs to $K[\widetilde C,x]$ and that both $x_0^t \cdot T(F_\alpha)$ and $H_\alpha$ are strongly reduced. Thus their difference $x_0^t \cdot T(F_\alpha)-H_\alpha$ is strongly reduced and moreover it belongs to $\mathfrak{I}_{J}$, because $x_0^t \cdot T(F_\alpha)-H_\alpha= -x_0^t \cdot F_\alpha+(x_0^t \cdot x^\alpha-H_\alpha)$. Hence, by Corollary \ref{last}, its $x$-coefficients belong to ${\mathfrak A}_{J}$ and in particular the coefficient of $x_0^t \cdot x^\gamma$ is of the type $C_{\alpha \gamma}- \phi_{\alpha \gamma}$, with $\phi_{\alpha \gamma}\in K[\widetilde C]$. Then $\mathfrak A_{J} \supseteq \mathfrak B$ and $\mathfrak A_{J}$ is generated by $\mathfrak B \cup \widetilde{ \mathfrak A}_{J}$.

To prove the second part, it is sufficient to show that ${\mathfrak A}_{J} \cap K[\widetilde C]=(\mathfrak D_1 \cup \mathfrak D_2)K[\widetilde C]$.

\lq\lq $\supseteq $\rq\rq\ Taking the $x$-coefficients in $x_0^t \cdot T(F_\alpha)-H_\alpha$ of monomials that are not divisible by $x_0^t$, we see that $\mathfrak A_{J}$ contains the $x$-coefficients of $H'_\alpha$. Then $\mathfrak A_{J} \cap K[\widetilde C] \supseteq \mathfrak D_1$, because $H'_\alpha \in K[\widetilde C,x]$.

Moreover we recall that $\mathfrak A_{J}$ is made by all the $x$-coefficients in the polynomials of $\mathfrak{I}_{J}$ that are strongly reduced. Indeed, $\mathfrak A_{J}$ is made by all the $x$-coefficients of the polynomials of $\mathfrak{I}_{J}$ that are $J$-reduced. Since the degree of the monomials in the variables $x$ of every polynomial in $\mathfrak{I}_{J}$ is $\geq m$, then \lq\lq$J$-reduced\rq\rq \ is equivalent to \lq\lq$\underline J$-reduced\rq\rq, that it is strongly reduced, by Lemma \ref{proptagliati}, (\ref{tagliv}). Then $\mathfrak A_{J} \cap K[\widetilde C]  \supseteq \mathfrak D_2$, because $(s{\mathcal G})K[\widetilde C,x] \subset \mathfrak{I}_{J}$.

\lq\lq $\subseteq $\rq\rq\  For every polynomial $F\in K[C,x]$, let us denote by $F^\phi$  the polynomial in $K[\widetilde C,x]$ obtained substituting every $C_{\alpha \gamma}\in C\setminus \widetilde C$ by $\phi_{\alpha \gamma}$; if $F$ is strongly reduced, then $F^\phi$ is strongly reduced too.
Observe that for every $x^\alpha \in B_{J}$ we have $F_\alpha ^\phi= x^\alpha-H''_\alpha$ and moreover $x_0^t(x^\alpha-H_\alpha'')-H'_\alpha \in (s\mathcal G)K[\widetilde C,x]$. In particular $x_0^t F_\alpha^\phi$ and $x_0^t(x^\alpha-H_\alpha'')-H'_\alpha$ are equal modulo $\mathfrak D_1$.

It remains to prove that every element $w\in {\mathfrak A}_{J} \cap K[ \widetilde C]$ can be obtained modulo $\mathfrak D_1$ as a $x$-coefficient in some strongly reduced polynomial of the ideal $(s{\mathcal G})\subset K[\widetilde C]$. We know that $w$ is a $x$-coefficient in a strongly reduced polynomial $D\in \mathfrak{I}_{J}$.

If $D=\sum D_\alpha F_\alpha\in \mathfrak{I}_{J}$, then for a suitable $t$,
\[
x_0^t\cdot D^\phi =\sum D^\phi _\alpha\cdot \left( x_0^t\cdot (x^\alpha -H''_\alpha)-H'_\alpha\right) \mod \mathfrak D_1
\]
and the polynomial in the right-hand side of the equality is strongly reduced and it belongs to $(s\mathcal G)K[\widetilde C,x]$. Therefore $w$ is still one of the $x$-coefficients of $D^\phi$ since it does not contain any variable in $C\setminus \widetilde C$ and it remains unchanged.  Then $w\in (\mathfrak D_1 \cup \mathfrak D_2)K[\widetilde C]$.
\end{proof}

 \begin{proposition}\label{formuleconsuperminimals}  Let $ \widetilde{\mathfrak A}_J $   be as in Theorem \ref{smallaffine} and let ${\mathfrak U} $ be any ideal in $K[{\widetilde C}]$.  Assume that  ${\mathfrak{ U} } \subseteq \widetilde{\mathfrak A}_J $  and that  the following conditions hold:
\begin{enumerate}[(i)]
\item \label{formuleconsuperminimals1} For every monomial $x^\beta \in B_J\setminus sB_J$, $x^\beta= x^{\underline{\alpha}} \ast_{\underline J} x^\delta \in J$, there exists $t$ such that we have a formula of type
\[
x_0^t\cdot x^\beta=\sum b_ix^{\eta_i} F_{\alpha_i}+H_\beta,
\]
with $a_i\in K[\widetilde C]$, $F_{\alpha_i}\in s\mathcal G$,  $x^{\eta_i}\leq_{Lex}x^\delta$, $x^{\underline{\eta}_j+\underline{\alpha}_j}= x^{\underline{\alpha}_j}\ast_{\underline J}x^{\underline{\eta}_j}$ and $H_\beta=H_\beta'+ x_0^t\cdot H_\beta''$, with $H_\beta$ strongly reduced, $x_0^t$ does not divide any monomial in $\Supp(H_\beta')$ and $\Coeff_x(H_\beta')\subseteq \mathfrak U$.
\item\label{formuleconsuperminimals2} For every polynomial $F_\alpha\in s\mathcal G$ and for every $x_i>\min(x^{\underline \alpha})$ there exists $t$ such that we have a formula of type
\[
x_0^t x_i F_\alpha=\sum b_j x^{\eta_j}F_{\alpha_j}+H_{i,\alpha}
\]
where $b_j\in K[\widetilde C]$, $F_{\alpha_j}\in s\mathcal G$, $x^{\eta_j}<_{\mathtt{Lex}}x_i$, $x^{\underline{\eta}_j+\underline{\alpha}_j}= x^{\underline{\alpha}_j}\ast_{\underline J}x^{\underline{\eta}_j}$, $H_{i,\alpha}$ strongly reduced and $\Coeff_x(H_{i,\alpha})\subseteq \mathfrak U$.
\end{enumerate}
Then $\mathfrak U=(\mathfrak D_1\cup\mathfrak D_2)=\widetilde{\mathfrak{A}}_J$.
\end{proposition}

\begin{proof}
Thanks to (\ref{formuleconsuperminimals1}), we immediately have that $\mathfrak D_1\subseteq \mathfrak U$.

For the inclusion $\mathfrak D_2\subseteq \mathfrak U$, observe that if (\ref{formuleconsuperminimals1}) and (\ref{formuleconsuperminimals2}) hold for $\mathfrak U$, then we can use the same arguments of Proposition \ref{prop:formula} and obtain that:\\
for every $F_\alpha \in s\mathcal G$, for every $x^\delta$, there exists $t$ such that
\begin{equation}\label{eqsuper}
x_0^tx^\delta F_\alpha=\sum b_j x^{\eta_j}F_{\alpha_j}+H
\end{equation}
with $b_j\in K[\widetilde C]$, $F_{\alpha_j}\in s\mathcal G$, $x^{\eta_j}<_{\mathtt{Lex}}x^\delta$, $x^{\underline{\eta}_j+\underline{\alpha}_j}= x^{\underline{\alpha}_j}\ast_{\underline J}x^{\underline{\eta}_j}$ $\Supp_x(H_{\delta,\alpha})\subseteq \cN(J)$ and $\Coeff_x(H_{\delta,\alpha})\subseteq \mathfrak U$.

We can also prove the uniqueness of such a rewriting: thanks to the uniqueness of the decomposition by $\ast_{\underline J}$, the polynomials $x^{\eta_j}F_{\alpha_j}$ that can appear in \eqref{eqsuper} have pairwise different head terms.
So an analogous of Corollary \ref{cor:formula} holds for this setting.

Thanks to this uniqueness, as in Corollary \ref{last}, we get the non trivial inclusion of the thesis.
\end{proof}

Proposition \ref{formuleconsuperminimals} is very important from the computational point of view. Indeed, its condition (\ref{formuleconsuperminimals1}) allows to explicitely construct the set of polynomials $\mathfrak B$, namely to write a $J$-marked set $\widetilde{\mathcal G}$ in $K[\widetilde C,x]$, whose superminimal set is $s\mathcal G$. Using such a $J$-marked set in $K[\widetilde C,x]$, we can use either Theorem \ref{nostrobuch} or Corollary \ref{nostrobuchEK}  or Theorem \ref{paolo} to obtain a set of generators for $\widetilde{\mathfrak A}_J$.
 For instance,  the algorithm presented in the Appendix is based on Theorem \ref{paolo} and the proof of its correctness on Proposition  \ref{formuleconsuperminimals}.
In the future, we will  investigate which  is the best set of polynomials to start from in order to get a performing algorithm for the computation of equations for $\Mf(J)$. The correctness of such an algorithm will be verified by  the conditions of Proposition  \ref{formuleconsuperminimals}.

\subsection{\texorpdfstring{Relations among $\Mf({\underline J}_{\geq m})$ as $m$ varies}{Relations among marked families of m-truncation ideals as m varies}}

In this subsection we will compare the marked schemes constructed from different truncations of a saturated strongly stable ideal $\underline J$.  Let us consider two  integers $m',m$, ($m'<m$). If $I$ is an ideal in the $\underline{J}_{\geq m'}$-marked family $\Mf(\underline{J}_{\geq m'})$, then it is not difficult to show that $I_{\geq m}$ belongs to the marked family $\Mf(\underline{J}_{\geq m})$, namely that there is a injective map of sets $\Mf(\underline{J}_{\geq m'})\rightarrow \Mf(\underline{J}_{\geq m})$.  Aim of the present subsection is a scheme theoretical version of this fact; indeed we will prove that there is a closed embedding of schemes $\Mf(\underline{J}_{\geq m'})\hookrightarrow \Mf(\underline{J}_{\geq m})$ that induces the previous one on the sets of closed points.
It is sufficient to prove the existence of such a closed embedding for $m'=m-1$; in this case we denote the embedding map by $\phi_m$. Furthermore, we characterize the cases in which $\phi_m$ is a isomorphism.

To this purpose, the main tool we will use is the set of defining equations for a $J$-marked scheme obtained by superminimal reduction, namely the ideal $\widetilde{\mathfrak A}_J$; moreover  we will consider at the Zariski tangent space of $\Mf(\underline J_{\geq m})$ at the origin, denoted by $T_0(\Mf(\underline J_{\geq m}))$.

\begin{remark}\label{tangente}
As for any affine variety, if $\Mf(J)$ is defined by an ideal $\mathfrak{U}$ as a subscheme of an affine space $\Af^{N}$, then the Zariski tangent space $T_0(\Mf(J))$ is defined by the linear part of a set generators of $\mathfrak U$ so that it can be identified to a linear subspace of $\Af^{N}$. In the special case of marked schemes, it is quite easy to compute a set of generators for $T_0(\Mf(J))$, using the properties and techniques of \cite[Definition 3.4 and Proposition 4.3]{LR}, \cite[Proposition 3.4]{RT} and \cite[Theorem 3.2]{FR}.
\end{remark}

Theorem \ref{schisom} is inspired by an analogous result proved for Gr\"obner Strata in \cite[Theorem 4.7]{LR}. Given a monomial ideal $J$, the \emph{Gr\"obner Stratum} $\mathcal St(J,\prec)$ of $J$ w.r.t. a term order $\prec$ can be isomorphically projected in its Zariski tangent space at the origin $T_0(\mathcal St(J,\prec))$ (see \cite[Proposition 4.3]{LR}). Furthermore, if the origin is a smooth point, then $\mathcal St(J,\prec)$ is isomorphic to this tangent space. Unluckily, it is not true that for every strongly stable ideal $J$ there exists a term order $\prec$ such that $\Mf(J)\simeq \mathcal St(J,\prec)$, as shown in \cite[Appendix]{CR}, so in general we cannot project isomorphically $\Mf(J)$ into $T_0(\Mf(J))$.

 We introduce some useful notations: once fixed a saturated strongly stable ideal $\underline J$ and a positive integer $m$, we denote by
\begin{itemize}
\item $\mathcal G^{[m]}$ a $\underline J_{\geq m}$-marked set as in \eqref{JbaseC} and with $F_\beta^{[m]}$ a marked polynomial belonging to $\mathcal G^{[m]}$;
\item $C^{[m]}$ the set of parameters $C_{\alpha\gamma}^{[m]}$ appearing in the tails of the marked polynomials in $\mathcal G^{[m]}$;
\item $s\mathcal G^{[m]}$ the set of superminimals of $\mathcal G^{[m]}$;
\item $\widetilde{C}^{[m]}$ the subset of $C^{[m]}$ containing only the parameters $\widetilde{C}_{\alpha\gamma}^{[m]}$ appearing in the tails of marked polynomials in $s\mathcal G^{[m]}$;
\item $\widetilde{\mathfrak A}^{[m]}$ is the ideal  in $K[\widetilde C^{[m]}]$  defining $\Mf(\underline J_{\geq m})$ as a subscheme in $\mathbb A^{\vert\widetilde C^{[m]}\vert}$ (defined in Theorem  \ref{smallaffine}).
\end{itemize}

\begin{theorem}\label{schisom}
Let $\underline J$ be a saturated strongly stable ideal and let $m$ be any integer.
With the previous notations, the followings hold:
\begin{enumerate}[(i)]
\item \label{schisom_i}
$\Mf(\underline J_{\geq m-1})$ is a closed subscheme of $\Mf(\underline J_{\geq m})$ cut out by a suitable linear space. More precisely, $\widetilde{C}^{[m-1]}$ can be identified with a suitable subset of $\widetilde{C}^{[m]}$ so that the following diagram of schemes commutes:

\begin{equation}\label{diagramma}
\begin{tikzpicture}
\node (1) at (0,0) [] {$\Mf(\underline J_{\geq m-1})$};
\node (2) at (3,0) [] {$\Mf(\underline J_{\geq m})$};
\node (3) at (0,-1.5) [] {$\Af^{\vert \widetilde{C}^{[m-1]}\vert}$};
\node (4) at (3,-1.5) [] {$\Af^{\vert \widetilde{C}^{[m]}\vert}$};
\draw [right hook->] (1) --node[above]{\scriptsize $\phi_m$} (2);
\draw [right hook->] (1) -- (3);
\draw [right hook->] (2) -- (4);
\draw [right hook->] (3) -- (4);
\end{tikzpicture}
\end{equation}

\item \label{schisom_iiii}
Let $\Omega$ be the number of monomials $x^\alpha \in B_{\underline J}$ of degree $m+1$ divisible by $x_1$ and $\Theta:=\vert B_{\underline J}\cap S_{\leq m-1}\vert$; then,
\[\dim T_0(\Mf(\underline J_{\geq m}))\geq \dim T_0(\Mf(\underline J_{\geq m-1}))+ \Omega\cdot \Theta.\]
\item \label{schisom_iiiii}
$\Mf(\underline J_{\geq m-1})\simeq\Mf(\underline J_{\geq m})$ if  and only if either $\underline J_{\geq m-1}= \underline J_{\geq m}$ or no monomial of degree $m+1$ in $B_{\underline J}$ is divisible by $x_1$.
\end{enumerate}
In particular:
\[\Mf(\underline J_{\geq \rho-1})\simeq  \Mf(\underline J_{\geq m}), \text{ for every } m\geq \rho\]
where $\rho$ is the maximal degree of  monomials divisible by $x_1$ in $B_{\underline J}$.
\end{theorem}

\begin{proof}\
\begin{enumerate}[(i)]
\item
Thanks to Theorem \ref{smallaffine}, a $\underline J_{\geq m}$-marked scheme is defined by an ideal generated by polynomials of $K[\widetilde C^{[m]}]$ that are constructed using only the superminimals. So, now it is enough to prove that the set of superminimals $s\mathcal G^{[m-1]}$ corresponds to $s\mathcal G^{[m]}$ modulo a subset of the variables $\widetilde C^{[m]}$, in the following sense.

Consider $x^\alpha \in sB_{\underline J_{\geq m-1}}$. If $\vert \alpha\vert\geq m$, then $x^\alpha$ belongs to $sB_{\underline J_{\geq m}}$ and we can identify $F^{[m]}_\alpha \in s\mathcal G^{[m]}$ and $F^{[m-1]}_\alpha \in s\mathcal G^{[m-1]}$ (and in particular the variables in their tails: $\widetilde C^{[m]}_{\alpha \gamma}=\widetilde C^{[m-1]}_{\alpha \gamma}$).

If $\vert\alpha\vert=m-1$, then we can consider the corresponding superminimal element $F^{[m]}_\beta \in s\mathcal G^{[m]}$, with $x^\beta=x_0\cdot x^\alpha$. Then we identify the variable $\widetilde C_{\beta\delta'}^{[m]}$, which is the coefficient of a monomial in $\Supp_x(F_\beta^{[m]})$ of kind $x^{\delta'}=x_0\cdot x^\delta$, with the variable $\widetilde C_{\alpha \delta}^{[m-1]}$ which  is the coefficient of the  monomial $x^\delta$ in $\Supp_x(F_\alpha^{[m-1]})$.

We repeat this identifications for all $x^\alpha \in sB_{\underline J_{\geq m-1}}$ and we denote by $\overline {C}^{[m]}$ the subset of $\widetilde C^{[m]}$ containing the variables non-identified with variables of $\widetilde C^{[m-1]}$, that is the variables appearing as coefficients of monomials  not divisible by $x_0$ in the tails of polynomials in $s\mathcal G^{[m]}\setminus s\mathcal G^{[m-1]}$. Now, every polynomial in $s\mathcal G^{[m]}\mod (\overline C^{[m]})$ either belongs to $s\mathcal G^{[m-1]}$ or is a polynomials of $s\mathcal G^{[m-1]}$ multiplied by $x_0$. Thanks to Theorem \ref{smallaffine}, we have that
\begin{equation*}\widetilde{\mathfrak A}^{[m]}+\left(\overline C^{[m]}\right)\simeq \widetilde{\mathfrak A}^{[m-1]}.\end{equation*}
This relation among the ideals induces the embeddings of scheme of diagram \eqref{diagramma}.
\item
We now consider $x^\gamma \in B_{\underline J}$, $\vert \gamma\vert=m+1$, $x^\gamma$ divisible by $x_1$. We define $x^\beta:=x^\gamma/x_1$; observe that $x^\beta \in \cN(\underline J)$. Furthermore, $x^\beta$ is not divisible by $x_0$, otherwise $x^\gamma$ would be too.

Then, for every $x^{\underline\alpha} \in B_{\underline J}$ with $\vert\underline\alpha\vert\leq m-1$, there is $F_{\alpha}^{[m]}=x^{\underline\alpha} x_0^{m-\vert\underline\alpha\vert}-T(F_{\alpha}^{[m]}) \in s\mathcal G^{[m]}$ such that $x^\beta \in \Supp_x(T(F_{\alpha}^{[m]}))$. We focus on the coefficient $\widetilde C_{\alpha\beta}^{[m]}$ of $x^\beta$. Since $x^\beta$ is not divisible by $x_0$, $\widetilde C_{\alpha\beta}^{[m]}$ cannot be identified with a coefficient appearing in $F_{\alpha}^{[m-1]}=x^{\underline\alpha} x_0^{m-\vert\underline\alpha\vert-1}-T(F_{\alpha}^{[m-1]})\in s\mathcal G^{[m-1]}$. So $\widetilde C_{\alpha\beta}^{[m]}$ belongs to the subset of variables $\overline C^{[m]}$ defined in the proof of (\ref{schisom_i}).

We now use the construction of $T_0(\Mf(\underline J_{\geq m}))$ recalled in Remark \ref{tangente}. If we think about syzygies of the ideal $\underline J_{\geq m}$, we can see that in a $S$-polynomial, $F_{\alpha}^{[m]}$ is multiplied by a monomial $x^\delta$ divisible by $x_i$, $i>0$. In particular, $x^\delta\cdot x^\beta$ belongs to $\underline J_{\geq m}$: if $x_i=x_1$ we are done by construction, otherwise we apply the strongly stable property because $x_1 x^\beta x^\delta=
\frac{x^\gamma}{x_i}x_1x^\delta$ belongs to $\underline J_{\geq m}$. This means that the coefficient $\widetilde C_{\alpha\beta}^{[m]}$ does not appear in any equation defining $T_0(\Mf(\underline J_{\geq m}))$.

Applying this argument to the $\Omega$ monomials in $B_{\underline J}$ of degree $m+1$ which are divisible by $x_1$ and to the $\Theta$ monomials in $B_{\underline J}$ of degree $\leq m-1$, we obtain the result.

\item
If $\underline J_{\geq m}=\underline J_{\geq m-1}$, obviously $\Mf(\underline J_{\geq m})=\Mf(\underline J_{\geq m-1})$.
We now assume that $\underline J_{\geq m}\neq \underline J_{\geq m-1}$ and no monomial of degree $m+1$ in the monomial basis of $\underline J$  is divisible by $x_1$; we prove that every polynomial in $s\mathcal G^{[m]}$ either belongs to $s\mathcal G^{[m-1]}$ or it is the product of $x_0$ by the \lq\lq corresponding\rq\rq\  polynomial in $s\mathcal G^{[m-1]}$.

If $x^\alpha \in sB_{\underline J_{\geq m-1}}$ and $\vert\alpha\vert \geq m$, then $F^{[m]}_\alpha \in s\mathcal G^{[m]}$ and $F^{[m-1]}_\alpha \in s\mathcal G^{[m-1]}$ have the same shape and we can identify them letting $\widetilde C^{[m]}_{\alpha \gamma}=\widetilde C^{[m-1]}_{\alpha \gamma}$, as done in the proof of (\ref{schisom_i}). If $\vert\alpha\vert =m-1$, then  $x^\beta=x_0\cdot x^\alpha \in sB_{\underline J_{\geq m}}$ and all the monomials in the support of $x_0\cdot F^{[m-1]}_\alpha$ appear in the support of $F^{[m]}_\beta$ (and we identify their coefficients as above). In the support of $F^{[m]}_\beta$ there are  also some more monomials that are not divisible by  $x_0$. We will prove now that the coefficients of these last monomials in fact belong to $\widetilde{\mathfrak A}^{[m]}$.

Consider the monomial $x_0\cdot x_1\cdot x^\alpha$. If we perform its reduction using $s\mathcal G^{[m]}$, the first step of reduction will lead to
\[
x_0\cdot x_1\cdot x^\alpha\xrightarrow{\ s\mathcal G^{[m]}\ast\ }x_1T(F^{[m]}_\beta).
\]
Let $x^\gamma$ be a monomial of $\Supp( T(F^{[m]}_\beta))$. If $x_1\cdot x^\gamma \in \underline J_{\geq m}$, then $x_1\cdot x^\gamma =x^{\underline {\alpha}'}\ast_{\underline J} x^\eta$, with $x^{\underline {\alpha}'}\in B_{\underline J}$ and $x^\eta<_{\mathtt{Lex}}x_1$.
If $x^\eta=1$, then $\vert\underline{\alpha}'\vert=m+1$ and $x^{\underline{\alpha}'}$ is divisible by $x_1$, against the hypothesis.
Then $x^\eta=x_0^{\overline t}$, with $t>0$, and so the monomial $x_1\cdot x^\gamma\in \underline J_{\geq m}$ is actually divisible by $x_0$.
If $x_1\cdot x^\gamma \in \cN(\underline J_{\geq m})$, then this monomial is not further reducible, so that its coefficient belongs to $\widetilde {\mathfrak A}^{[m]}$.

Vice versa, by contradiction suppose now that $\underline J_{\geq m-1}\neq \underline J_{\geq m}$ and that there exists $x^\alpha \in B_{\underline J}$ divisible by $x_1$, $\vert\alpha\vert=m+1$.
Using (\ref{schisom_iiii}), we have that $T_0(\Mf(\underline J_{\geq m-1}))\not \simeq T_0(\Mf(\underline J_{\geq m}))$ because $\dim T_0(\Mf(\underline J_{\geq m-1}))<\dim T_0(\Mf(\underline J_{\geq m}))$, and so $\Mf(\underline J_{\geq m-1}) \not \simeq \Mf(\underline J_{\geq m})$.
\end{enumerate}

For the last part of the statement, note that if $\rho$ is the maximal degree of a monomial divisible by $x_1$ in the monomial basis of $\underline J$, for every $m\geq \rho$, applying iteratively (\ref{schisom_iiiii}) we obtain
\[
\Mf(\underline J_{\geq \rho-1})\simeq \Mf(\underline J_{\geq m}).\qedhere
\]
\end{proof}

In the above setting, if $p(t)$ is the Hilbert polynomial of $S/\underline J$ and $r$ is its Gotzmann number, it is worth  considering the $r$-truncation of $\underline J$. Indeed, in \cite{BLR} the authors prove that $\Mf(\underline J_{\geq r})$ is naturally isomorphic to an open subset of the Hilbert scheme $\hilbp$. We recall that $r$ is the maximum among the regularities of ideals that are closed points of $\hilbp$; hence, $r\geq \reg(\underline J)\geq \rho-1$.
Then Theorem \ref{schisom} allows us to study such an open subset of $\hilbp$ embedded in an affine space of lower dimension than the expected one. More precisely:
\begin{corollary}\label{cornumpar}
Let $\underline J$ be a saturated strongly stable ideal, and let $\rho$ be the maximal degree of  monomials divisible by $x_1$ in $B_{\underline J}$.
For every  $m\geq \rho-1$, $\Mf(\underline J_{\geq m})$ can be embedded in an affine space of dimension
\begin{equation}\label{numpar}
\vert\widetilde C^{[\rho-1]}\vert=\sum_{x^{\alpha}\in sB_{\underline J_{\geq \rho-1}}}\vert \cN(\underline J)_{\vert\alpha\vert}\vert\leq \vert B_{\underline J}\vert \cdot p(r'),
\end{equation}
where $r'=\reg(\underline J)$ and $p(t)$ is the Hilbert polynomial of $S/\underline J$.
\end{corollary}

\begin{proof}
The equality of \eqref{numpar} directly follows from Theorem \ref{schisom}.
For the inequality, we simply need to observe that the regularity  $r'$ of a strongly stable ideal is simply the maximum of the degrees of its monomial generators; hence every monomial in $sB_{\underline J_{\geq \rho-1}}$ has degree $\leq r'$. Furthermore $r'$ is greater than or equal to the regularity of the Hilbert function of $S/\underline J$, thus $\vert \cN(\underline J)_{\vert\alpha\vert}\vert\leq \cN(\underline J_{r'})=p(r')$.

An equivalent proof follows from the diagram \eqref{diagramma} of Theorem \ref{schisom}.
\end{proof}

\section{Examples}\label{secexamples}
In the hypothesis that the field $K$ has characteristic 0, the methods of computations developed in the previous sections  can be applied to the study of Hilbert schemes: indeed, for $m$ big enough, $\Mf(\underline J_{\geq m})$ corresponds to an open subset of the Hilbert scheme parameterizing the ideals having the same Hilbert polynomial as $S/\underline J$ (see \cite{BLR}).

Now we give some examples for applications of the obtained results, mainly Theorem \ref{nostrobuch}, Theorem \ref{paolo} and Theorem \ref{schisom}.  We keep on using the notations introduced before Theorem \ref{schisom}.

\begin{example}\label{espunti}
Let $\underline J$ be the saturated strongly stable ideal $(x_n,\dots,x_2,x_1^{\mu})\subseteq S=K[x_0,\dots,x_n]$. Observe that $\underline J$ is a ${\mathtt{Lex}}$-segment, the Hilbert polynomial of $S/\underline J$ is $p(t)=\mu$, the regularity of $\underline J$ is $r'=\mu$, and also $\rho=\mu$.
By Corollary \ref{cornumpar}, $\Mf(\underline J)$ can be embedded into an affine space of dimension $2n-2+\mu$. Using the criterion of Theorem \ref{nostrobuch}, we can see that actually $\Mf(\underline J)\simeq \Af^{2n-2+\mu}$, as shown also  in \cite{RT}. By Theorem \ref{schisom}, (\ref{schisom_iiiii}), $\Mf(\underline J_{\geq m})$ is isomorphic to $\Mf(\underline J)$  for every $m\leq \mu-2$.

It is well-known that $\mathrm{Proj}\, (S/\underline J)$ is the {\tt Lex}-point of $\mathcal Hilb^n_\mu$ and lies on a component of dimension $n\mu$ (see \cite{RS}). Then $\Mf(\underline J_{\geq m})$ is not isomorphic to an open subset of $\mathcal Hilb^n_\mu$ for every $m\leq \mu -2$.

On the other hand, the same reasonings above leads to  $\Mf(\underline J_{\geq \mu})\simeq \Af^{n\mu}$ so that $\Mf(\underline J_{\geq \mu})$ is  an open subset of $\mathcal Hilb^n_\mu$. This is shown also  in \cite{RT}.
\end{example}

\begin{example}
We consider the strongly stable  saturated ideal
\[
\underline J=(x_2^3, x_1x_2^2, x_1^2x_2, x_1^5)\subseteq K[x_0,x_1,x_2].
\]
It corresponds to a point of $\mathcal Hilb^2_{8}$, with Gotzmann number $r=8$, the regularity   $r'$ of $\underline J$ is 5, and the same value for $\rho$.
By Theorem \ref{schisom}  we have that $\Mf(\underline J_{\geq 4})\simeq\Mf(\underline J_{\geq r})$.
Observe that $\underline J_{\geq 4}$ is not segment w.r.t. any term order (see \cite[Appendix]{CR}), hence in this case the results of \cite{LR} do not apply.

In \cite[Appendix]{CR}, the authors first consider {$\Mf(\underline J_{\geq 4})$} as an affine subscheme of $A^{64}$ and then show that $45$ of the variables can be eliminated, but using a time-consuming process of elimination of variables. By Corollary \ref{cornumpar}, we can directly embed $\Mf(\underline J_{\geq r})$ in an affine space of dimension $32$, and we have to eliminate only $13$ of the remaining variables.
\end{example}

\begin{example}\label{ex:4tP3}
We take $p(t)=4t$, $n=3$, $q(t)=\binom{3+t}{3} -p(t) = \binom{3+t}{3}-4t$; the Gotzmann number of $p(t)$ is $r=6$. The Hilbert scheme $\hilb^3_{4t}$ can be considered as a subscheme of the Grassmannian $\mathbb G=\mathbb G(q(6),K[x]_6)$  of linear spaces of dimension $q(6)=60$ in the vector space $K[x]_6$ of dimension $\binom{3+6}{3}=84$ (see \cite[Section 1]{BLR} for some details about this construction). Therefore equations for $\hilb^3_{4t}$ involve  $E=\binom{84}{60}-1\sim 6\cdot 10^{20}$ \Pl\ coordinates. We can obtain an open cover of $\hilb^3_{4t}$ by the non-vanishing of each \Pl\ coordinate of $\mathbb G$: we get $E$ open subsets, each of them  isomorphic to a subscheme of $\Af^{1440}$.

In \cite{BLR} the authors consider a different open cover (up to the action of $PGL(4)$) of $\hilb^3_{4t}$, formed by 4 open subsets only, isomorphic to a marked scheme of a suitable truncation of the saturated strongly stable monomial ideals   $\underline J_i$, $i=1,2,3, 4$ in $K[x]$. We can choose for every $i$ the truncation $m=r=6$, nevertheless in order to perform computations with a lower number of variables, it is better to choose the truncations according to Theorem \ref{schisom}.

We denote the cardinality of the monomial basis of $\underline J_i$ by $\sigma_i$.  In the following table we list the dimensions of the different affine spaces where we can embed the marked schemes, using Theorem \ref{schisom} and Corollary \ref{cornumpar}.
 \[
\begin{array}{|c|c|c|c|c|c|c|}
\cline{2-7}
 \multicolumn{1}{c|}{} & \phantom{\Big\vert}\text{Monomial basis of } \underline{J}_i  & \reg(\underline J_i) & \sigma_i &  \rho_i-1 &\sigma_i p\big(\reg(\underline{J}_i)\big) &\left\vert \widetilde C^{[\rho_i-1]}\right\vert\\
 \hline
  \underline J_1 & \phantom{\Big\vert} x_3^2,x_3x_2,x_2^3\phantom{\Big\vert}& 3 & 3 & -1& 36 & 28\\
 \hline
  \underline J_2&\phantom{\Big\vert} x_3^2,x_3x_2,x_3x_1^2 ,x_2^4\phantom{\Big\vert} & 4 & 4 &2& 64 & 44\\
 \hline
 \underline J_3& \phantom{\Big\vert} x_3^2,x_3x_2,x_3x_1,x_2^5,x_2^4x_1\phantom{\Big\vert} & 5& 5 & 4&100 & 88\\
 \hline
 \underline J_4& \phantom{\Big\vert} x_3,x_2^5,x_2^4x_1^2\phantom{\Big\vert} & 6 & 3 &5 & 72 & 64\\
 \hline
\end{array}
\]

Observe that for $\underline J_1$ and $\underline J_2$, the truncation giving an open subset of $\hilb^3_{4t}$ is exactly the saturated ideal.

$\underline J_4$ is the ${\mathtt{Lex}}$-segment ideal: $\Mf((\underline J_4)_{\geq 5})$ is isomorphic to $\Af^{23}$ (see \cite[Theorem 7.3]{LR}). In this case we should further eliminate $41$ variables.
This means that  our bounds are not in general \lq\lq sharp\rq\rq, however the computational consequences of Theorem \ref{schisom} are significant and our results allow the  treatment of non-trivial cases that cannot be handled with \lq\lq classical\rq\rq\ techniques.
\end{example}

\section*{Appendix. A pseudocode description of the algorithm for computing a $J$-marked scheme}

We now describe a prototype of the algorithm for computing $J$-marked families based on Proposition \ref{formuleconsuperminimals} and on the analogous of the set $L_1$ defined in Theorem \ref{paolo}. The ideal $J$ is always supposed to be a strongly stable $m$-truncation ideal.

Let us suppose that the following functions are made available.
\begin{itemize}
\item $\textsc{generators}(J)$. It determines the monomial basis of $J$.
\item $\textsc{superminimalGenerators}(J)$. It determines the superminimal generators of $J$.
\item $\textsc{superminimalReduction}(H,sG)$. Given a $J$-marked superminimal set $sG$ and a polynomial $H$, it returns a pair $(t,h)$ where $t$ is the minimal power of $x_0$ such that there is a  superminimal reduction of $x_0^t H$ to a strongly reduced polynomial and $h$ is such polynomial, namely $x_0^t H \SGred h$ (as in Theorem \ref{ridsm}, (\ref{ridsm-ii})).
\item $\textsc{quotientAndRemainder}(H,t)$. Given a polynomial $H$ and a non-negative integer $t$, it returns the pair of polynomials $(H',H'')$ such that $H = H' + H''x_0^t$.
\item \textsc{pairsL1}$(sG)$. Given a $J$-marked superminimal set $sG$, it computes the pairs of polynomials belonging to the set $\mathcal{L}_1$ (analogous of the set $L_1$ of Theorem \ref{paolo}).
\item \textsc{coeff}$(H,x^\alpha)$. It returns the coefficient of the monomial $x^\alpha$ in the polynomial $H$ (obviously 0 if $x^\alpha \notin \Supp(H)$).
\end{itemize}

\begin{algorithm}[H]
\begin{algorithmic}[1]
\STATE $\textsc{markedScheme}(J)$
\REQUIRE $J \subset K[x_0,\ldots,x_n]$ strongly stable $m$-truncation ideal.
\ENSURE an ideal defining the marked scheme $\Mf(J)$.
\STATE $B_J \leftarrow \textsc{generators}(J)$;
\STATE $sB_J \leftarrow \textsc{superminimalGenerators}(J)$;
\STATE $\widetilde{\mathcal{G}} \leftarrow \emptyset$; \quad $s\mathcal{G} \leftarrow \emptyset$;
\FORALL{$x^\alpha \in sB_J$}
\STATE $F_{\alpha} \leftarrow x^\alpha$;
\FORALL{$x^\beta \in \cN(J)_{\vert\alpha\vert}$}
\STATE $F_{\alpha} \leftarrow F_{\alpha} + \widetilde{C}_{\alpha\beta} x^\beta$;
\ENDFOR
\STATE $\widetilde{\mathcal{G}} \leftarrow \widetilde{\mathcal{G}}\cup \{F_{\alpha}\}$;
\STATE $s\mathcal{G} \leftarrow s\mathcal{G} \cup \{F_{\alpha}\}$;
\ENDFOR
\STATE $\textsf{equations} \leftarrow \emptyset$;
\STATE $\overline{B}_J \leftarrow B_J \setminus sB_J$;
\FORALL{$x^\alpha \in \overline{B}_J$}\label{start1}
\STATE $(t,H) \leftarrow \textsc{superminimalReduction}(x^\alpha,s\mathcal{G})$;
\STATE $(H',H'') \leftarrow \textsc{quotientAndRemainder}(H,t)$;
\FORALL{$x^\eta \in \Supp(H')$}
\STATE $\textsf{equations} \leftarrow \textsf{equations} \cup \{\textsc{Coeff}(H',x^\eta)\}$;
\ENDFOR
\STATE $\widetilde{\mathcal{G}} \leftarrow\widetilde{\mathcal{G}} \cup \{x^\alpha - H''\}$;
\ENDFOR\label{stop1}
\STATE\label{start2} $\mathcal{L}_1 \leftarrow \textsc{pairsL1}(s\mathcal{G})$;
\FORALL{$(F_{\alpha},F_{\alpha'}) \in \mathcal{L}_1$}
\STATE\label{SMredALG} $(t,H) \leftarrow \textsc{SuperminimalReduction}\big(S(F_{\alpha},F_{\alpha'}),s\mathcal{G}\big)$;
\FORALL{$x^\eta \in \Supp(H)$} \label{eq2start}
\STATE $\textsf{equations} \leftarrow \textsf{equations} \cup \{\textsc{Coeff}(H,x^\eta)\}$;
\ENDFOR\label{eq2stop}
\ENDFOR\label{stop2}
\RETURN $(\textsf{equations})$;
\end{algorithmic}
\end{algorithm}

\begin{algCorrect}
The algorithm \textsc{markedScheme} is correct.
\end{algCorrect}

\begin{proof}
To prove that the algorithm terminates it is sufficient to recall that   the superminimal reduction is Noetherian (Theorem \ref{ridsm} (\ref{ridsm-i})).

Now we show that the  algorithm \textsc{markedScheme} returns a set of generators for the ideal defining $\Mf(J)$. The starting point is the  $J$-marked  superminimal set $s\mathcal G$ given  in Definition \ref{decomposizione}  , having parameters in $\widetilde C$ as coefficients of every monomial in the tails and get a set  $\textsf{equations}$ of polynomials in $K[\widetilde C]$.  We claim that the ideal $\mathfrak{U}$ generated by  $\textsf{equations}$ coincides with the ideal $\widetilde{ \mathfrak{A}}_J$ of Theorem \ref{smallaffine}, by Proposition \ref{formuleconsuperminimals}.

Indeed in the first part (lines \ref{start1}-\ref{stop1}), the algorithm computes the superminimal reduction $H$ of each monomial $x^\alpha \in B_J\setminus sB_J$ and it imposes the conditions required by Proposition \ref{formuleconsuperminimals} (\ref{formuleconsuperminimals1}), in other words the algorithm computes the set $\mathfrak D_1\subseteq \widetilde{\mathfrak A}_J$ of Definition \ref{elim}. At the same time, the algorithm constructs the $J$-marked set $\widetilde{\mathcal G}\subset K[\widetilde C,x]$.

In the second part (lines \ref{start2}-\ref{stop2}), the algorithm considers pairs of superminimal generators  $(F_\alpha,F_{\alpha'})$ such that  $x_ix^\alpha=x^{\underline \alpha'}\ast_{\underline J}x^\eta$. Recall that $x^\eta <_\mathtt{Lex} x_i$ by Lemma \ref{descLex}. These couples of polynomials in $s\mathcal G$ correspond to the couples of the set $L_1$ in Theorem \ref{paolo}.

At line \ref{SMredALG} of the algorithm we compute the superminimal reduction of the associated $S$-polynomial
\[
x_0^t S(F_{\alpha},F_{\alpha}') = x_0^{t} \left(x_i x_0^{t'} F_{\alpha} -  x^\eta x_0^{t''} F_{\alpha'}\right) \SGredPar H \qquad x_i x_0^{t'} = \frac{\text{lcm}(x^\alpha,x^{\alpha'})}{x^\alpha},\ x^\eta x_0^{t''}= \frac{\text{lcm}(x^\alpha,x^{\alpha'})}{x^{\alpha'}}
\]
that is applying Lemma \ref{lexrid}
\[
x_0^{t}(x_0^{t'} x_i F_{\alpha} - x_0^{t''} x^\eta F_{\alpha'}) - \sum b_jx^{\eta_j}F_{\beta_j} =  H
\]
with $b_j \in K[\widetilde{C}]$, $F_{\beta_j} \in s\mathcal{G}$, $x^{\eta_j}<_{\mathtt{Lex}}x_i x_0^{t'}$ and $x^{\underline{\eta_j}}<_{\mathtt{Lex}}x_i$, so that
\[
x_0^{\overline{t}} x_i F_{\alpha} = x^{\overline\eta} F_{\alpha'} +\sum b_jx^{\eta_j}F_{\beta_j} + H.
\]
The polynomial $H$ is strongly reduced and it belongs to the ideal $({s\mathcal G})\subseteq K[\widetilde C,x]$, then its $x$-coefficients belong to $\mathfrak D_2\subseteq\widetilde{\mathfrak{A}}_J$.

Then by construction (lines \ref{eq2start}-\ref{eq2stop}), $\mathfrak U$ is contained in $\widetilde{\mathfrak A}_J$ and it satisfies the condition required by Proposition \ref{formuleconsuperminimals} (\ref{formuleconsuperminimals2}), hence $\mathfrak U=\widetilde{\mathfrak A}_J$.
\end{proof}

We are convinced that this version can be strongly strengthened drawing inspiration from some of the improvements studied for the computation of Gr\"obner bases and border bases. In this direction, we have already developed  a first prototype which is giving good and promising results. In the following table, we report the results of the computation of the marked schemes considered in Example \ref{ex:4tP3}. The algorithm has been run on a MacBook Pro with a 2,4 GHz Intel Core 2 Duo processor.
\[
\begin{array}{|c|c|c|c|}
\cline{2-4}
\multicolumn{1}{c|}{\phantom{\Big\vert}} & \text{Parameters} & \text{Equations} & \text{Time}\\
\hline
\phantom{\Big\vert}\Mf(\underline J_1) & 28 & 28 & 0.165568 \text{ seconds} \\
\hline
\phantom{\Big\vert}\Mf(\underline{J}_2) & 44 & 64 & 0.295802 \text{ seconds}\\
\hline
\phantom{\Big\vert}\Mf\big((\underline{J}_3)_{\geq 4}\big) & 88 & 228 &  110050 \text{ seconds}\\
\hline
\end{array}
\]
 The prototype of the algorithm is available at \\
\href{http://www.personalweb.unito.it/paolo.lella/HSC/Documents/MarkedSchemes.m2}{\texttt{www.personalweb.unito.it/paolo.lella/} \texttt{HSC/Documents/MarkedSchemes.m2}}

\section*{Acknowledgments}

The first, third and fourth authors were supported by the PRIN \lq\lq Geometria delle variet\`{a} algebriche e dei loro spazi di moduli\rq\rq, cofinanced by MIUR (Italy) (cofin 2008). The second author was supported by the PRIN \lq\lq Geometria Algebrica e Aritmetica, Teorie Coomologiche e Teoria dei Motivi\rq\rq, cofinanced by MIUR (Italy) (cofin 2008) and by FARO (Finanziamento per l'Avvio di Ricerche Originali), Polo delle Scienze e delle Tecnologie, Univ. di Napoli Federico II, co-financed by Compagnia San Paolo (2009).

\bibliographystyle{plain}

\end{document}